\documentclass{amsart}

\usepackage{amssymb,amsxtra,latexsym}

\usepackage[full]{textcomp}
\usepackage[T1]{fontenc}
\usepackage[osf]{newtxtext}
\usepackage[bigdelims]{newtxmath}
\usepackage[scaled=0.94,scr=rsfs]{mathalfa}
\usepackage[scaled=0.93]{cabin}
\usepackage[varqu,varl]{inconsolata}

\usepackage{booktabs}

\usepackage{bm} 

\usepackage{tikz-cd}
\tikzcdset{arrow style=math font}
\tikzcdset{diagrams={nodes={inner sep=3pt}}}

\usepackage[pagebackref]{hyperref}

\definecolor{ceruleanblue}{rgb}{0.16, 0.32, 0.75}
\hypersetup{colorlinks=true,allcolors=ceruleanblue}

\usepackage{multicol}

\makeatletter

\def\theglossary{\@restonecoltrue\if@twocolumn\@restonecolfalse\fi
\columnseprule\z@ \columnsep 35\p@
\let\@makessectionhead\indexsec
\@xp\section\@xp*\@xp{\glossaryname}%
\let\item\@idxitem
\parindent\z@ \parskip\z@\@plus.3\p@
\relax\footnotesize}

\def\glossaryname{Notation Index}

\makeatother

\makeglossary

\renewcommand\theenumi{\roman{enumi}}
\renewcommand\labelenumi{(\theenumi)}

\newcommand\note[1]%
{$^\dag$\marginpar{\footnotesize{$^\dag${#1}}}}

\divide\count253 by 60 
\divide\count255 by 60
\multiply\count255 by 60 
\advance\count254 by -\count255 

\def\today{\number\year-\ifnum\month<10
0\fi\number\month-\ifnum\day<10 0\fi\number\day}

\def\hour{\ifnum\count253<10
0\number\count253\else\number\count253\fi}

\def\minute{\ifnum\count254<10
0\number\count254\else\number\count254\fi}

\numberwithin{equation}{subsection}

\swapnumbers

\newtheorem{theorem}[equation]{Theorem}
\newtheorem{introtheorem}{Theorem}
\newtheorem{lemma}[equation]{Lemma}
\newtheorem{proposition}[equation]{Proposition}
\newtheorem{corollary}[equation]{Corollary}

\theoremstyle{definition}
\newtheorem{definition}[equation]{Definition}
\newtheorem{example}[equation]{Example}

\newtheorem{remark}[equation]{Remark}

\newtheorem*{notation*}{Notation}

\newcommand\lie{\mathfrak}

\newcommand\liea{\lie{a}}
\newcommand\liec{\lie{c}}
\newcommand\g{\lie{g}}
\newcommand\h{\lie{h}}
\newcommand\liek{\lie{k}}
\newcommand\liel{\lie{l}}
\newcommand\m{\lie{m}}
\newcommand\n{\lie{n}}
\newcommand\p{\lie{p}}

\newcommand\bb[1]{{\text{\bf#1}}}

\newcommand\N{\bb{N}}
\newcommand\Z{\bb{Z}} 

\newcommand\R{\bb{R}} 
\newcommand\C{\bb{C}}

\newcommand\M{\bb{M}}

\newcommand\V{\bb{V}}
\newcommand\W{\bb{W}}

\newcommand\ca{\mathscr}

\DeclareMathOperator\Aut{Aut}
\DeclareMathOperator\Ad{Ad}
\DeclareMathOperator\ad{ad}
\DeclareMathOperator\codim{codim}

\DeclareMathOperator\Diff{Diff}
\DeclareMathOperator\End{End}
\DeclareMathOperator\ev{ev}

\DeclareMathOperator\Hom{Hom}
\DeclareMathOperator\Hol{Hol}
\DeclareMathOperator\id{id}

\DeclareMathOperator\Inn{Inn}

\DeclareMathOperator\Lie{Lie}
\DeclareMathOperator\Mon{Mon}
\DeclareMathOperator\MQK{MQK}
\DeclareMathOperator\Out{Out}
\DeclareMathOperator\Pf{Pf}
\DeclareMathOperator\pr{pr}
\DeclareMathOperator\rank{rank}

\DeclareMathOperator\Stab{Stab}
\DeclareMathOperator\stab{stab}
\DeclareMathOperator\supp{supp}

\newcommand\group[1]{{\text{\bf#1}}}

\newcommand\SO{\group{SO}}

\newcommand\A{\group{A}}
\newcommand\B{\group{B}}

\newcommand\D{\group{D}}
\newcommand\E{\group{E}}
\newcommand\FF{\group{F}}

\newcommand\abs[1]{\lvert#1\rvert}

\newcommand\inner[1]{\langle#1\rangle}

\newcommand\qu[1][\kern.3ex]{/\kern-.7ex/_{\kern-.4ex#1}}
\newcommand\bigqu[1][\,\,]{\big/\kern-.85ex\big/_{\!\!#1}}

\newcommand\powl{[\kern-.3ex[}
\newcommand\powr{]\kern-.3ex]}
\newcommand\bigpowl{\bigl[\kern-.6ex\bigl[}
\newcommand\bigpowr{\bigr]\kern-.6ex\bigr]}

\newcommand\inj{\hookrightarrow}
\newcommand\sur{\mathrel{\to\kern-1.8ex\to}}
\newcommand\iso{\mathrel{\hookrightarrow\kern-1.8ex\to}}

\newcommand\longto{\longrightarrow}

\newcommand\longhookrightarrow{\lhook\joinrel\longrightarrow}

\newcommand\longinj{\longhookrightarrow}
\newcommand\longsur{\mathrel{\longrightarrow\kern-1.8ex\to}}

\newcommand\dyet{\text{\it\dj}}
\newcommand\Dyet{\text{\it\DJ}}

\newcommand\zerodots%
{\smash{\makebox[0pt]{$_{\lower8pt\hbox{\Large$0$}}%
\ddots^{\lower3pt\hbox{\Large$0$}}$}}}
\newcommand\bigzerodots%
{\smash{\makebox[0pt]{$_{\lower6pt\hbox{\Large$0$}}%
\ddots^{\hbox{\Large$0$}}$}}}

\newcommand\F{{\ca{F}}}
\newcommand\X{{\lie{X}}}
\newcommand\barX{\lie{X}\kern-0.55em\overline{\phantom I}\kern0.15em}
\newcommand\barF{\F\kern-0.6em\overline{\phantom{\rm F}}}

\newcommand\cc{{\mathrm{c}}}
\newcommand\cv{{\mathrm{cv}}}
\newcommand\ct{{\mathrm{ct}}}
\newcommand\even{\text{\rm even}}
\newcommand\odd{\text{\rm odd}}
\newcommand\hor{\text{\rm hor}}
\newcommand\ver{\text{\rm vert}}
\newcommand\hhor[1]{\text{\rm$#1$-hor}}
\newcommand\bas{\text{\rm bas}}
\newcommand\bbas[1]{\text{\rm$#1$-bas}}
\newcommand\CE{\text{\rm CE}}
\newcommand\BRST{\text{\rm BRST}}
\newcommand\LC{\text{\rm LC}}
\newcommand\prin{\text{\rm prin}}
\newcommand\tot{\text{\rm tot}}
\newcommand\uni{\text{\rm uni}}

\begin{document}


\title[Localization for transverse Lie algebra actions on Riemannian
  foliations]{Cohomological localization for transverse Lie algebra
  actions on Riemannian foliations}




\author{Yi Lin}

\address{Department of Mathematical Sciences, Georgia Southern
  University, Statesboro, GA 30460 USA}

\email{yilin@georgiasouthern.edu}

\author{Reyer Sjamaar}

\address{Department of Mathematics, Cornell University, Ithaca, NY
14853-4201, USA}

\email{sjamaar@math.cornell.edu}

\subjclass[2010]{57R30 (57R91 58A12)}

\date\today


\begin{abstract}
We prove localization and integration formulas for the equivariant
basic cohomology of Riemannian foliations.  As a corollary we obtain a
Duistermaat-Heckman theorem for transversely symplectic foliations.
\end{abstract}


\maketitle

\tableofcontents


\section{Introduction}

Most of the recent advances in the area of Riemannian foliations, such
as the index theorems of Br\"uning et
al.~\cite{bruening-kamber-richardson;index-basic-foliations} and
Gorokhovsky and Lott~\cite{gorokhovsky-lott;index-transverse-molino},
and the cohomological localization formulas of
T\"oben~\cite{toeben;localization-basic-characteristic} and Goertsches
et al.~\cite{goertsches-nozawa-toeben;chern-simons-foliations}, build
on the structure theory developed by Molino in his
monograph~\cite{molino;riemannian-foliations}.  The central feature of
this structure theory is that the leaf closures of a Riemannian
foliation, in sharp contrast to those of a general foliation, form a
singular foliation.  This leaf closure foliation possesses a type of
``internal'' symmetry in the shape of a locally constant sheaf of Lie
algebras of vector fields, known as the Molino structure sheaf or
centralizer sheaf, which acts transversely on each leaf in such a way
that the orbit of the leaf is the closure of the leaf.  A Killing
foliation is a Riemannian foliation whose centralizer sheaf is
globally constant, in which case it is automatically abelian, as shown
by Molino.  The Goertsches-Nozawa-T\"oben cohomological localization
formulas hold in the setting of Killing foliations.  They are
integration formulas for differential forms that are basic with
respect to the foliation and equivariant with respect to the Molino
centralizer algebra.

Many Riemannian foliations possess additional, ``external'' Lie
algebras of symmetries that commute with the Molino centralizer sheaf,
and the goal of this paper is to extend the cohomological localization
formulas to this wider class of symmetries.  Our main results are two
localization theorems in equivariant basic de Rham cohomology.  The
first is a contravariant Borel-Atiyah-Segal version, which resolves a
problem posed
in~\cite{goertsches-nozawa-toeben;chern-simons-foliations}.

\begin{introtheorem}
Let $M$ be a compact manifold equipped with a Riemannian foliation
$\F$ and an isometric transverse action of an abelian Lie algebra
$\g$.  Then the inclusion of the fixed leaf set $M^\g$ into $M$
induces an isomorphism in localized equivariant basic cohomology
$S^{-1}H_\g(M,\F)\cong S^{-1}H_\g(M^\g,\F)$.
\end{introtheorem}

Here $S$ is the multiplicative subset $S\g^*\backslash\{0\}$ of the
symmetric algebra $S\g^*$.  See Theorem~\ref{theorem;borel} for a more
precise version.  Our second main result is a covariant
Atiyah-Bott-Berline-Vergne localization theorem.

\begin{introtheorem}
Let $M$ be a compact manifold equipped with a Riemannian foliation
$\F$ and an isometric transverse action of an abelian Lie algebra
$\g$.  Suppose that the foliation $\F$ is transversely oriented and
that the Molino structure bundle has trivial determinant.  Let
$\alpha$ be a $d_\g$-closed equivariant basic form on $M$.  Then
\[
\fint_M\alpha=\sum_X\fint_X\eta_X^{-1}\wedge i_X^*\alpha,
\]
where the sum is over the connected components $X$ of the fixed-leaf
manifold $M^\g$.
\end{introtheorem}

Here $\fint$ denotes transverse integration and $\eta_X$ is the
equivariant basic Euler form of the component $X$.  See
Theorem~\ref{theorem;abbv} for a fuller statement.  These results rely
on a version of the equivariant Thom isomorphism theorem which we
established in~\cite{lin-sjamaar;thom}.  An immediate consequence of
these theorems is a Duistermaat-Heckman theorem for the transversely
symplectic case, Theorem~\ref{theorem;duistermaat-heckman}.  Another
application is the computation of the basic Betti numbers of toric
quasifolds in~\cite[\S\,6]{lin-yang;basic-kirwan}.  We postpone
further applications of our theorems to a sequel to this paper.  See
also~\cite{casselmann-fisher;localization-k-contact} for related
recent work.

Lie algebras in this paper act on foliated manifolds not as true
vector fields, but as transverse vector fields, which are equivalence
classes of vector fields.  Because of this one cannot form an action
Lie algebroid in the usual way.  However, there is a way of
amalgamating the Lie algebra with the foliation to create what we call
the transverse action Lie algebroid.  Just like action algebroids and
foliation algebroids, transverse action algebroids are always
integrable (Proposition~\ref{proposition;transverse-action-groupoid}).
Although this fact plays no further role in the paper, we have
included it because it may offer an approach to cohomological
localization formulas for foliations that are not Riemannian and to
which the Molino structure theory does not apply.  Namely, the
equivariant basic de Rham complex is a subcomplex of the much larger
simplicial de Rham complex of any Lie groupoid integrating the
transverse action algebroid, the properties of which remain to be
explored.

Another subsidiary result that we hope may be useful elsewhere is an
orbit type stratification theorem for isometric transverse Lie algebra
actions on Riemannian foliations, Theorems~\ref{theorem;symmetry}
and~\ref{theorem;semicontinuous}, which is almost the same as for
proper Lie group actions, even though the slice theorem fails in our
context.

A notation index is provided at the end of the paper.

\section{Preliminaries}\label{section;preliminary}

All manifolds in this paper will be smooth ($C^\infty$) and
paracompact and all vector fields and foliations will be smooth.  In
this preliminary section we start by reviewing some elementary
foliation theory, largely for the purpose of introducing notation.
See also the notation index in the back.  Then we recall the notion of
a transverse Lie algebra action on a foliated manifold.  This is not a
Lie algebra action on the manifold, but should be thought of as a
model for a Lie algebra action on the leaf space of the foliation.  We
introduce a certain integrable Lie algebroid associated with a
transverse action, which we call the \emph{transverse action Lie
  algebroid}, and which does not seem to have appeared in the
literature before.  We show that it is integrable to a not necessarily
Hausdorff Lie groupoid.  We finish by reviewing Goertsches and
T\"oben's equivariant basic de Rham complex.

\subsection{Foliations}\label{section;foliation}

Let $M$ be a manifold and let $\F$ be a foliation of $M$.  We denote
by $\F(x)$ the leaf through $x\in M$, by $T_x\F$ the tangent space to
the leaf at $x$, and by $T\F$ the integrable subbundle of $TM$ tangent
to $\F$.  We call a subset $X$ of the manifold \emph{$\F$-invariant}
or \emph{$\F$-saturated} if for every $x\in X$ the leaf $\F(x)$ is
contained in $X$.  The \emph{normal bundle} of the foliation is the
vector bundle $N\F=TM/T\F$ over $M$.  A smooth map $f\colon M\to M'$
to a second foliated manifold $(M',\F')$ is said to be \emph{foliate}
if it maps each leaf of $\F$ to a leaf of $\F'$,
i.e.\ $f(\F(x))\subseteq\F'(f(x))$ for all $x\in M$.  Thus $f$ is
foliate if and only if the tangent map $Tf\colon TM\to TM'$ maps $T\F$
to $T\F'$.  The tangent map then descends to a bundle map $Nf\colon
N\F\to N\F'$, which we refer to as the \emph{normal derivative}.
\glossary{MF@$(M,\F)$, foliated manifold}
\glossary{F@$\F$, foliation}
\glossary{Fx@$\F(x)$, leaf of $x$}
\glossary{TFx@$T_x\F$, tangent space to leaf at $x$}
\glossary{TF@$T\F$, tangent bundle of $\F$}
\glossary{NF@$N\F$, normal bundle of $\F$}
\glossary{Nf@$Nf$, normal derivative of foliate map $f$}

Let $\X(M)=\Gamma(TM)$ be the Lie algebra of all vector fields on $M$
and let $\X(\F)=\Gamma(T\F)$ be the space of sections of $T\F$.  Then
$\X(\F)$ is the Lie subalgebra of $\X(M)$ consisting of all vector
fields tangent to the leaves of the foliation.  We denote the
normalizer $N_{\X(M)}(\X(\F))$ of $\X(\F)$ in $\X(M)$ by $\lie{N}(\F)$
and we call elements of $\lie{N}(\F)$ \emph{foliate vector fields}.  A
vector field $w$ on $M$ is foliate if and only if the flow of $w$
consists of foliate maps if and only if for every vector field $v$
tangent to $\F$ the commutator $[v,w]$ is also tangent to $\F$.  We
call elements of the quotient Lie algebra
$\X(M,\F)=\lie{N}(\F)/\X(\F)$ \emph{transverse vector fields}.  A
transverse vector field is not a vector field, but an equivalence
class of foliate vector fields modulo vector fields tangent to $\F$.
These Lie algebras, which go by various different names in the
standard references~\cite{kamber-tondeur;foliated-bundles},
\cite{moerdijk-mrcun;foliations-groupoids},
\cite{molino;riemannian-foliations},
and~\cite{tondeur;geometry-foliations}, form a short exact sequence
\begin{equation}\label{equation;foliate-transverse}
\X(\F)\longinj\lie{N}(\F)\longsur\X(M,\F).
\end{equation}
The transverse vector fields form a subspace of the space of sections
of the normal bundle,
\[
\X(M,\F)= \lie{N}(\F)/\X(\F)\subseteq\Gamma(N\F)= \X(M)/\X(\F).
\]
Being the quotient of a Lie algebra by a subalgebra, the space
$\Gamma(N\F)$ is an $\X(\F)$-module.  The subspace $\X(M,\F)$ is equal
to $\Gamma(N\F)^{\X(\F)}$, the space of $\X(\F)$-fixed sections of
$N\F$.  Thus transverse vector fields are $\X(\F)$-fixed sections of
the normal bundle $N\F$.
\glossary{G.amma@$\Gamma$, smooth global sections of bundle}
\glossary{XM@$\X(M)$, vector fields on $M$}
\glossary{NNF@$\lie{N}(\F)$, foliate vector fields on $(M,\F)$}
\glossary{XMF@$\X(M,\F)$, transverse vector fields on $(M,\F)$}
\glossary{XF@$\X(\F)$, vector fields tangent to $\F$}

If the foliation $\F$ is \emph{strictly simple}, i.e.\ given by the
fibres of a surjective submersion $M\to P$ with connected fibres, then
the leaf space $M/\F$ is a manifold diffeomorphic to~$P$.  In this
case the Lie algebra of transverse vector fields $\X(M,\F)$ is nothing
but the Lie algebra of vector fields on the leaf space.  If the
foliation is not strictly simple, we regard transverse vector fields
as a substitute for vector fields on the leaf space.  We can also
regard the leaf space as an \'etale stack and $\X(M,\F)$,
following~\cite[\S\,5.7]{hoffman-sjamaar;hamiltonian-stack}, as the
Lie algebra of vector fields on this stack.

Given a foliate map $f\colon(M,\F)\to(M',\F')$, we say that two
transverse vector fields $v\in\X(M,\F)$ and $v'\in\X(M',\F')$ are
\emph{$f$-related}, and we write
\begin{equation}\label{equation;related}
v\sim_fv',
\end{equation}
if the normal derivative of $f$ satisfies $N_xf(v_x)=v'_{f(x)}$ for
all $x\in M$.
\glossary{.@$\sim_f$, $f$-relatedness between vector fields}

The $\X(\F)$-module structure on the space of sections $\Gamma(N\F)$,
\[
\nabla^\F\colon\X(\F)\times\Gamma(N\F)\longto \Gamma(N\F),
\]
is induced by the Lie bracket $\X(\F)\times\X(M)\to\X(M)$ and is
called the \emph{Bott connection} or \emph{partial connection} of the
foliation.  It satisfies the usual rules
\[
\nabla^\F_{fu}v=f\nabla^\F_uv,\qquad
\nabla^\F_u(fv)=L(u)(f)v+f\nabla^\F_uv
\]
for all smooth functions $f$ and sections $u\in\X(\F)$,
$v\in\Gamma(N\F)$.  The partial connection provides a \emph{partial
  horizontal lifting map} $T\F\to T(N\F)$, i.e.\ a splitting over
$T\F$ of the surjection $T(N\F)\to TM$.  For every leaf $L$ of $\F$
the partial connection induces a genuine flat connection on the
restricted bundle $N\F|_L$.
%
\glossary{.@$\nabla^\F$, partial connection of foliation}
%

\subsection{Transverse Lie algebra actions}\label{section;transverse}

Let $\g$ be a finite-dimensional real Lie algebra and let $(M,\F)$ be
a foliated manifold.  A \emph{foliate action} of $\g$ on $(M,\F)$ is a
Lie algebra homomorphism $\g\to\lie{N}(\F)$.  By the Lie-Palais
theorem~\cite[Ch.\ 4]{palais;lie-transformation}, if the vector fields
induced by $\g$ are complete, a foliate $\g$-action integrates to a
smooth $G$-action, where $G$ is a suitable Lie group with Lie algebra
$\g$.  The group $G$ then acts by foliate diffeomorphisms and we refer
to this as a \emph{foliate $G$-action}.
\glossary{g@$\g$, Lie algebra}
\glossary{G@$G$, Lie group}

A \emph{transverse action} of $\g$ on $(M,\F)$ is a Lie algebra
homomorphism $a\colon\g\to\X(M,\F)$.  For $\xi\in\g$ we denote the
transverse vector field $a(\xi)$ by $\xi_M$.  Each of these is
represented by a foliate vector field $\tilde{\xi}_M\in\lie{N}(\F)$.
However, there is in general no way of choosing the representatives
$\tilde{\xi}_M$ in such a way that they generate a foliate
$\g$-action.  In other words a transverse $\g$-action does not
necessarily lift to a foliate $\g$-action, and in particular there is
no guarantee that a transverse $\g$-action ``integrates'' to a group
action.  Nevertheless,
Proposition~\ref{proposition;transverse-action-groupoid} below tells
us that a transverse Lie algebra action always integrates to the
action of a certain Lie groupoid.  We regard transverse Lie algebra
actions, which are the focus of our interest, as a substitute for Lie
group actions on the leaf space $M/\F$.
\glossary{x.i@$\xi_M$, transverse vector field induced by $\xi\in\g$}

For the remainder of \S\,\ref{section;transverse} we fix a transverse
$\g$-action $a\colon\g\to\X(M,\F)$.

If the foliation is strictly simple, then the transverse action
amounts to an ordinary $\g$-action on the manifold $M/\F$.  Similarly,
if $S$ is a transversal to the foliation, the transverse action on $M$
restricts to an ordinary $\g$-action on $S$, because every $x\in S$
has an open neighbourhood $U$ such that the foliation $\F|_U$ is
strictly simple with leaf space $S\cap U$.  Indeed, since transverse
vector fields on $M$ are $\X(\F)$-invariant sections of $N\F$, the
transverse $\g$-action on $U$ is uniquely determined by the
$\g$-action on $S\cap U$.

A \emph{Lie algebroid} over $M$ is a real vector bundle $\liea\to M$
equipped with a vector bundle map $t\colon \liea\to TM$ called the
\emph{anchor} and a real Lie algebra structure
\[
[{\cdot},{\cdot}]\colon\Gamma(\liea)\times\Gamma(\liea)\longto
\Gamma(\liea)
\]
on the space of smooth sections which satisfies the Leibniz rule
\begin{equation}\label{equation;leibniz}
[s_1,fs_2]=f[s_1,s_2]+L(t(s_1))(f)s_2
\end{equation}
for all sections $s_1$, $s_2\in\Gamma(\liea)$ and functions $f\in
C^\infty(M)$.  Here $L(t(s_1))(f)$ denotes the Lie derivative of $f$
along the vector field $t(s_1)$.

If the transverse $\g$-action on $M$ lifts to a true Lie algebra
action $\g\to\X(M)$, we can form the action Lie algebroid $\g\ltimes
M$ as in~\cite[\S\,6.2]{moerdijk-mrcun;foliations-groupoids}.  For a
general transverse $\g$-action the definition of an action Lie
algebroid does not make sense, but must be modified as follows.

Let $\g_M$ be the trivial bundle with fibre $\g$ over $M$.  We define
the \emph{transverse action Lie algebroid} $\g\ltimes\F$ of the
transverse $\g$-action to be the fibred product
\[
\g\ltimes\F=\g_M\times_{N\F}TM=\{\,(x,\xi,v)\in M\times\g\times
TM\mid\text{$v\in T_xM$ and $\xi_{M,x}=v\bmod T_x\F$}\,\},
\]
which is a smooth subbundle of the bundle $\g_M\times TM$ over $M$.
The projection $\g\ltimes\F\to\g_M$ defined by
$(x,\xi,v)\mapsto(x,\xi)$ has kernel $T\F$, so we have a short exact
sequence
\[
\begin{tikzcd}
T\F\ar[r,hook,"i"]&\g\ltimes\F\ar[r,two heads,"p"]&\g_M
\end{tikzcd}
\]
of vector bundles over $M$.   Define the
bundle map 
\begin{equation}\label{equation;anchor}
t\colon\g\ltimes\F\to TM
\end{equation}
by $t(\xi,v)=v$.  For every smooth map $\xi\colon M\to\g$ and every
$x\in M$ we define $\xi_x\in N_x\F$ to be the value of the transverse
vector field $a(\xi(x))\in\X(M,\F)$ at $x$.  A smooth section of
$\g\ltimes\F$ is a pair $(\xi,v)\in C^\infty(M,\g)\times\X(M)$
satisfying $\xi_x=v_x\bmod T_x\F$ for all $x\in M$.  We define the
bracket of two sections $(\xi,v)$ and $(\eta,w)$ by
\begin{equation}\label{equation;bracket}
\bigl[(\xi,v),(\eta,w)\bigr](x)=
\bigl([\xi(x),\eta(x)]+L(v)(\eta)(x)-L(w)(\xi)(x),[v,w](x)\bigr)
\end{equation}
for $x\in M$.  We assert that this bracket makes $\g\ltimes\F$ a Lie
algebroid over~$M$.
\glossary{gF@$\g\ltimes\F$, transverse action Lie algebroid}

\begin{proposition}\label{proposition;transverse-action-lie-algebroid}  
The bracket~\eqref{equation;bracket} is the unique skew symmetric
$\R$-bilinear operation on\/ $\Gamma(\g\ltimes\F)$ that satisfies the
Leibniz rule~\eqref{equation;leibniz} with respect to the vector
bundle map~\eqref{equation;anchor} and has the property\/
$[(\xi,v),(\eta,w)]=([\xi,\eta],[v,w])$ for all constant maps $\xi$,
$\eta\in C^\infty(M,\g)$.  This operation is a Lie bracket on the
space of sections\/ $\Gamma(\g\ltimes\F)$ and hence makes\/
$\g\ltimes\F$ a Lie algebroid over $M$ with anchor~$t$.
\end{proposition}

\begin{proof}[Sketch of proof]
We omit the verification that~\eqref{equation;bracket} satisfies the
Leibniz rule.  The product Lie algebra $\g\times\lie{N}(\F)$ has a
subalgebra
\[
\lie{L}=\{\,(\xi,v)\in\g\times\lie{N}(\F)\mid\xi_M=v\bmod\X(\F)\,\}.
\]
The Lie algebra $\lie{L}$ naturally identifies with a subspace of
$\Gamma(\g\ltimes\F)$, namely by regarding an element $\xi\in\g$ as a
constant map $\xi\in C^\infty(M,\g)$.  On the subspace $\lie{L}$ the
bracket~\eqref{equation;bracket} coincides with the Lie bracket coming
from $\g\times\lie{N}(\F)$, namely
$[(\xi,v),(\eta,w)]=([\xi,\eta],[v,w])$.  Every section of
$\g\ltimes\F$ can be written as a finite sum $\sum_if_i(\xi_i,v_i)$
with $f_i\in C^\infty(M)$ and $(\xi_i,v_i)\in\lie{L}$.  In other
words, the space of sections $\Gamma(\g\ltimes\F)$ is generated by
$\lie{L}$ as a $C^\infty(M)$-module.  Armed with these facts one shows
that $\Gamma(\g\ltimes\F)$ is closed under the
bracket~\eqref{equation;bracket}, that~\eqref{equation;bracket} is the
unique skew symmetric bilinear map that satisfies the Leibniz rule and
coincides with the Lie bracket on $\lie{L}$, and that it satisfies the
Jacobi identity.
\end{proof}

A \emph{Lie groupoid} over $M$ is a (not necessarily Hausdorff)
manifold $\bb{G}$ equipped with surjective submersions $\bb{s}$,
$\bb{t}\colon\bb{G}\to M$, smooth maps $\bb{m}\colon\bb{G}_2\to\bb{G}$
(where $\bb{G}_2$ is the fibred product
$\{\,(f,g)\in\bb{G}\times\bb{G}\mid s(f)=t(g)\,\}$),
$\bb{i}\colon\bb{G}\to\bb{G}$, and $\bb{u}\colon M\to\bb{G}$, which
make $\bb{G}$ an abstract groupoid over $M$ with source map $\bb{s}$,
target map $\bb{t}$, multiplication law $\bb{m}$, inversion law
$\bb{i}$, and unit map $\bb{u}$.  If $\bb{G}$ is a Lie groupoid over
$M$, the normal bundle $\liea\to M$ of the unit bisection
$\bb{u}(M)\cong M$ is equipped with a natural Lie algebroid structure.
We say that a Lie algebroid $\liea$ obtained in this way is
\emph{integrable}, and that $\bb{G}$ \emph{integrates} $\liea$;
cf.~\cite[Ch.~6]{moerdijk-mrcun;foliations-groupoids}.

\begin{proposition}\label{proposition;transverse-action-groupoid}
The transverse action Lie algebroid $\g\ltimes\F$ is integrable.
\end{proposition}

\begin{proof}
The following method is adapted from~\cite[Theor\`eme
  2.1]{dazord;groupoide-holonomie}; see
also~\cite[\S\,6.2]{moerdijk-mrcun;foliations-groupoids}.  Let $G$ be
a Lie group with Lie algebra $\g$ and let $p\colon G\times M\to M$ be
the projection onto the second factor.  For every section $(\xi,v)\in
C^\infty(M,\g)\times\X(M)$ of $\g\ltimes\F$ define a vector field
$a_{G\times M}(\xi,v)\in\X(G\times M)$ by
\begin{equation}\label{equation;dazord}
a_{G\times M}(\xi,v)_{(g,x)}=(T_1L_g(\xi_x),v_x),
\end{equation}
where $L_g\colon G\to G$ denotes left multiplication by $g$.  The map
$a_{G\times M}\colon\Gamma(\g\ltimes\F)\to\X(G\times M)$ defines a
free action of the Lie algebroid $\g\ltimes\F$ on $G\times M$ in the
sense of~\cite[Definition
  3.1]{kossmann-mackenzie;actions-lie-algebroids}, so the vector
fields $a_{G\times M}(\xi,v)$ span a foliation $\ca{G}$ of $G\times
M$.  The left multiplication action of $G$ on the first factor of
$G\times M$ is $\ca{G}$-foliate, and therefore lifts naturally to a
$G$-action on the monodromy groupoid $\Mon(\ca{G})$, which is free and
proper.  The quotient manifold $\Mon(\ca{G})/G$ is a Lie groupoid over
$(G\times M)/G=M$, whose Lie algebroid is isomorphic to $\g\ltimes\F$.
\end{proof}
\glossary{Mon@$\Mon(\F)$, monodromy groupoid}

Let $\liea$ be a Lie algebroid over $M$ with anchor $t\colon\liea\to
TM$ and let $x$ be a point of $M$.  We denote by $\liea_x$ the fibre
of $\liea$ at $x$, by $\stab(\liea,x)=\ker(t_x\colon\liea_x\to T_xM)$
the isotropy or stabilizer of $x$, and by $\liea(x)$ the $\liea$-orbit
of $x$, i.e.\ the leaf through $x$ of the singular foliation
$t(\liea)\subseteq TM$.  The stabilizer
$\stab(\liea,x)\subseteq\liea_x$ inherits a Lie algebra structure from
$\liea$.  The orbit $\liea(x)$ is an immersed submanifold of $M$ with
tangent space $T_x(\liea(x))=t(\liea_x)$ at $x$.
\glossary{stab@$\stab(\liea,x)$, stabilizer Lie algebra of $x$ with
  respect to Lie algebra or Lie algebroid $\liea$}
\glossary{a@$\liea(x)$, orbit of $x$ with respect to Lie algebroid
  $\liea$}
\glossary{a@$\liea_x$, fibre at $x$ of Lie algebroid $\liea$}

The stabilizer of $x$ for the transverse action Lie algebroid
$\liea=\g\ltimes\F$ is equal to
\[\stab(\g\ltimes\F,x)=\{\,\xi\in\g\mid\xi_{M,x}=0\,\}.\]
The image of the anchor map $t$ is a singular subbundle of $TM$, which
generates a singular foliation denoted by $\g{\cdot}\F$
in~\cite[\S\,2]{goertsches-toeben;equivariant-basic-riemannian}.  The
orbit $\g\ltimes\F(x)$ of $x$ under the transverse action Lie
algebroid is a leaf of this singular foliation, which we also refer to
as the \emph{$\g$-orbit of the leaf $\F(x)$}.  The next statement
shows that the stabilizers of different elements in the same
$\g\ltimes\F$-orbit are conjugate under the adjoint group of $\g$.
\glossary{gFx@$\g\ltimes\F(x)$, orbit of $x$ under transverse action
  Lie algebroid}

\begin{lemma}\phantomsection\label{lemma;stabilizer}
\begin{enumerate}
\item\label{item;same}
Suppose $x$ and $y\in M$ are in the same leaf of $\F$.  Then
$\stab(\g\ltimes\F,x)=\stab(\g\ltimes\F,y)$.
\item\label{item;conjugate}
Let $x\in M$ and $\xi\in\g$.  Choose a foliate vector field
$\tilde{\xi}_M\in\lie{N}(\F)$ representing $\xi_M\in\X(M,\F)$.  Let
$\phi_t$ denote the flow of $\tilde{\xi}_M$ and suppose that
$\phi_t(x)$ exists for all $t\in[0,1]$.  Let $y=\phi_1(x)$ and let
$e^{\ad_\xi}$ denote the exponential of $\ad_\xi$ in the adjoint group
$\Ad(\g)$.  Then
$\stab(y,\g\ltimes\F)=e^{\ad_\xi}\bigl(\stab(x,\g\ltimes\F)\bigr)$.
\end{enumerate}
\end{lemma}

\begin{proof}
\eqref{item;same}~Let $\xi\in\g$.  Then $\xi_M$ is a transverse vector
field, so $\nabla_v\xi_M=0$ for all $v\in\X(\F)$.  Therefore
$\xi_{M,\phi_t(x)}=\phi_{t,*}(\xi_{M,x})$, where $\phi_t$ denotes the
flow of $v$.  Now suppose $\xi\in\stab(\g\ltimes\F,x)$,
i.e.\ $\xi_{M,x}=0$.  Then $\xi_{M,\phi_t(x)}=0$.  Since this holds
for all $v\in\X(\F)$, we have $\xi_{M,y}=0$,
i.e.\ $\xi\in\stab(\g\ltimes\F,y)$.

\eqref{item;conjugate}~Let $G$ be a Lie group with Lie algebra $\g$.
The path $\gamma(t)=(\exp_G(t\xi),\phi_t(x))$ in $G\times M$ is the
trajectory through $(1,x)$ of the vector field
$a(\xi,\tilde{\xi}_M)=(\xi_L,\tilde{\xi}_M)$ defined
in~\eqref{equation;dazord}, where $\xi_L$ is the left-invariant vector
field associated with $\xi$.  It is therefore tangent to the foliation
$\ca{G}$ defined in the proof of
Proposition~\ref{proposition;transverse-action-groupoid}.  Its
homotopy class $[\gamma]$ is an element of the monodromy groupoid
$\Mon(\ca{G})$ with source $\gamma(0)=(1,x)$ and target
$\gamma(1)=(g,y)$, where $g=\exp_G(\xi)$.  Therefore the image of
$[\gamma]$ in the quotient groupoid $\Mon(\ca{G})/G$ is an arrow $f$
with source $x$ and target $y$.  It follows that
$\Stab(y)=f\cdot\Stab(x)\cdot f^{-1}$, where $\Stab(x)$ is the
stabilizer group of $x$ in the quotient groupoid.  Passing to the Lie
algebras we get
\[
\stab(y,\g\ltimes\F)=\Ad_g\bigl(\stab(x,\g\ltimes\F)\bigr)=
e^{\ad_\xi}\bigl(\stab(x,\g\ltimes\F)\bigr).\qedhere
\]
\end{proof}
\glossary{Stab@$\Stab(G,x)$, stabilizer group of $x$ with respect to
  Lie group or Lie groupoid $G$}

The transverse $\g$-action is called \emph{free at $x$} if
$\stab(\g\ltimes\F,x)=0$, and \emph{free} if it is free at all $x\in
M$.  We say that $x\in M$ is \emph{$\g\ltimes\F$-fixed} if
$\stab(\g\ltimes\F,x)=\g$.  This means that for all $\xi\in\g$ the
transverse vector field $\xi_M$ vanishes at $x$, or equivalently, by
Lemma~\ref{lemma;stabilizer}\eqref{item;same}, that every foliate
representative of $\xi_M$ is tangent to the leaf $\F(x)$.  (So if the
$\g$-action integrates to a foliate $G$-action, the leaf $\F(x)$,
viewed as a point in the leaf space, is $G$-fixed.)  We call the set
\[M^\g=\{\,x\in M\mid\stab(\g\ltimes\F,x)=\g\,\}\]
the \emph{fixed-leaf set} of $M$.
\glossary{Mg@$M^\g$, fixed-leaf manifold}
%

\subsection{Basic differential forms}\label{section;basic}

Let $(M,\F)$ be a foliated manifold.  A differential form $\alpha$ on
$M$ is \emph{$\F$-basic} if its Lie derivatives $L(u)\alpha$ and
contractions $\iota(u)\alpha$ vanish for all vector fields
$u\in\X(\F)$.  The set of $\F$-basic forms is a differential graded
subalgebra of the de Rham complex $\Omega(M)$, which we denote by
$\Omega(M,\F)$.  Its cohomology is a graded commutative algebra called
the \emph{$\F$-basic de Rham cohomology} and denoted by $H(M,\F)$.  A
foliate map $f\colon(M,\F)\to(M',\F')$ induces a pullback morphism of
differential graded algebras $f^*\colon\Omega(M',\F')\to\Omega(M,\F)$
and hence a morphism of graded algebras $f^*\colon H(M',\F')\to
H(M,\F)$.  If the foliation is strictly simple, the basic de Rham
complex is just the de Rham complex of the manifold $M/\F$.
\glossary{O.megaMF@$\Omega(M,\F)$, basic de Rham complex}
%

\subsection{$\g$-differential graded algebras}
\label{section;differential-graded}

Let $\g$ be a finite-dimensional Lie algebra.  A
\emph{$\g$-differential graded module} (or \emph{$\g$-dgm}) is a
graded vector space $\M$ equipped with a linear map $d$ of degree $1$
and, for each $\xi\in\g$, a linear map $\iota(\xi)$ of degree $-1$ and
a linear map $L(\xi)$ of degree $0$ that depend linearly on $\xi$ and
satisfy the Cartan commutation rules
\begin{align*}
{[\iota(\xi),\iota(\eta)]}&=0,&[L(\xi),L(\eta)]&=L([\xi,\eta]),
&[d,d]&=0,\\
[L(\xi),d]&=0,&[\iota(\xi),d]&=L(\xi),&
[L(\xi),\iota(\eta)]&=\iota([\xi,\eta]),
\end{align*}
where the brackets denote graded commutators.  This notion was
introduced by Cartan~\cite{cartan;algebre-transgression} to axiomatize
the Chern-Weil theory of connections and curvature.  See
also~\cite[Appendice]{sergiescu;cohomologie-basique},
\cite{guillemin-sternberg;supersymmetry-equivariant},
\cite{alekseev-meinrenken;lie-chern-weil},
or~\cite[Appendix~A]{lin-sjamaar;thom}.  A \emph{morphism}
$\phi\colon\M\to\M'$ of $\g$-dgm $\M$ and $\M'$ is a degree $0$ map of
graded vector spaces that satisfies
$[\phi,d]=[\phi,\iota(\xi)]=[\phi,L(\xi)]=0$ for all $\xi\in\g$.  A
\emph{homotopy} between two morphisms $\phi_0$, $\phi_1\colon\M\to\M'$
of $\g$-dgm is a degree $-1$ map of graded vector spaces
$\kappa\colon\M\to\M'[-1]$ that satisfies $[\kappa,d]=\phi_1-\phi_0$
and $[\kappa,\iota(\xi)]=[\kappa,L(\xi)]=0$ for all $\xi\in\g$.
\glossary{Lie@$L$, Lie derivative}
\glossary{i.ota@$\iota$, contraction}
\glossary{MM@$\M$, $\g$-differential graded module}

An element $m$ of a $\g$-dgm $\M$ is \emph{$\g$-invariant} if
$L(\xi)m=0$ for all $\xi\in\g$, \emph{$\g$-horizontal} if
$\iota(\xi)m=0$ for all $\xi\in\g$, and \emph{$\g$-basic} if it is
$\g$-invariant and $\g$-horizontal.  The subspace of $\g$-basic
elements
\[
\M_\bbas{\g}=\{\,m\in\M\mid\text{$L(\xi)m=\iota(\xi)m=0$ for all
  $\xi\in\g$}\,\}
\]
is a subcomplex of $\M$.  The \emph{basic cohomology} of $\M$ is
defined as the cohomology of the complex $\M_\bbas{\g}$ and is denoted
by $H_\bbas\g(\M)$.
\glossary{Hgbas@$H_\bbas{\g}(\M)$, $\g$-basic cohomology of $\M$}

A \emph{$\g$-differential graded algebra} (or \emph{$\g$-dga}) is a
graded algebra $\A$ that is also a $\g$-dgm and for which the
structure maps $d$, $\iota(\xi)$, $L(\xi)$ are graded derivations.
\glossary{AA@$\A$, $\g$-differential graded algebra}
We say that a $\g$-dga $\A$ is \emph{locally free} if it admits a
\emph{connection}, i.e.\ a linear map $\theta\colon\g^*\to\A^1$
satisfying
\[
\iota(\xi)(\theta(x))=\inner{\xi,x}\in\A^0\quad\text{and}\quad
L(\xi)(\theta(x))=-\theta(\ad^*(\xi)x)
\]
for all $\xi\in\g$ and $x\in\g^*$.

The \emph{Weil algebra} $\W\g$ is a locally free commutative
$\g$-differential graded algebra that is universal in the following
sense: it is equipped with a connection $\theta_\uni$, and for every
locally free commutative $\g$-dga $\A$ and every connection $\theta$
on $A$ there is a unique morphism of $\g$-dga
$c_\theta\colon\W\g\to\A$ such that the diagram
\begin{equation}\label{equation;ccw}
\begin{tikzcd}
\g^*\ar[d,"\theta_\uni"']\ar[r,"\theta"]&\A
\\
\W\g\ar[ur,dashed,"c_\theta"']
\end{tikzcd}
\end{equation}
commutes.  The basic complex $(\W\g)_\bbas{\g}$ is isomorphic to
$(S\g^*)^\g$, the algebra of $\g$-invariant polynomials on $\g$,
equipped with the zero differential and with the generating subspace
$\g^*$ placed in degree~$2$.  The morphism $c_\theta$ induces a
morphism of graded algebras
\begin{equation}\label{equation;characteristic}
c_\theta\colon(S\g^*)^\g\longto\A_\bbas{\g},
\end{equation}
known as the \emph{characteristic} or \emph{Cartan-Chern-Weil
  homomorphism} of the connection $\theta$.

The \emph{Weil complex} $\M_\g$ of a $\g$-differential graded module
$\M$ is the basic complex of the $\g$-dgm $\W\g\otimes\M$,
\[\M_\g=(\W\g\otimes\M)_\bbas\g.\]
The \emph{equivariant cohomology} of $\M$ is defined as the cohomology
of the complex $\M_\g$ and is denoted by $H_\g(\M)$.  The equivariant
cohomology $H_\g(\M)$ is a module over $(\W\g)_\bbas{\g}=(S\g^*)^\g$.
A morphism $\phi\colon\M\to\M'$ of $\g$-dgm induces a morphism
$H(\phi)\colon H_\g(\M)\to H_\g(\M')$, and if two morphisms $\phi_0$
and $\phi_1$ are homotopic, then $H(\phi_0)=H(\phi_1)$.
\glossary{Wg@$\W\g$, Weil algebra of $\g$}
\glossary{MMgbas@$\M_\bbas{\g}(\M)$, $\g$-basic elements of $\M$}
\glossary{HgM@$H_\g(\M)$, equivariant cohomology of $\M$}

For our purposes the main point of
$\g$-differential graded algebras is the next result, which contains
as a special case the well-known fact that the $G$-equivariant de Rham
complex of a principal bundle with structure group $G$ is homotopy
equivalent to the de Rham complex of the base manifold.  (The latter
fact is true in general only if $G$ is compact, but the following
statement holds for arbitrary $\g$.)
See~\cite[Theorem~A.5.1]{lin-sjamaar;thom} for a proof.

\begin{theorem}\label{theorem;principal}
Let $\A$ be a locally free commutative $\g$-differential graded
algebra.  The morphism $j\colon\A\to\W\g\otimes\A$ defined by
$j(a)=1\otimes a$ is a homotopy equivalence with homotopy inverse
$\begin{tikzcd}[cramped,sep=small]\W\g\otimes\A\ar[r,"\simeq"]&\A
\end{tikzcd}$
defined by $x\otimes a\mapsto c_\theta(x)a$, where $\theta$ is a
connection on $\A$ and $c_\theta$ is as in~\eqref{equation;ccw}.
Hence $j$ induces a homotopy equivalence
$\begin{tikzcd}[cramped,sep=small]\A_\bbas{\g}\ar[r,"\simeq"]&
  (\W\g\otimes\A)_\bbas{\g}=\A_\g\end{tikzcd}$
and an isomorphism
$\begin{tikzcd}[cramped,sep=small]H_\bbas{\g}(\A)\ar[r,"\cong"]&
  H_\g(\A)\end{tikzcd}$.
\end{theorem}

We will require a relative version of this theorem.  Suppose that
$\g=\liek\times\h$ is the product of two Lie subalgebras $\liek$ and
$\h$, and let $\A$ be a $\g$-dga.  Suppose that $\A$ is locally free
as a $\liek$-dga and let $\theta\colon\liek^*\to\A^1$ be a
$\liek$-connection.  We say that the connection $\theta$ is
\emph{$\h$-invariant} if $\theta(\eta)\in\A^1$ is $\h$-invariant for
all $\eta\in\liek^*$.  The analogue of the map~\eqref{equation;ccw} in
this situation is a homomorphism
\begin{equation}\label{equation;equivariant-ccw}
c_{\h,\theta}\colon\W\liek\longto\W\h\otimes\A,
\end{equation}
which on basic elements gives the \emph{$\h$-equivariant
  Cartan-Chern-Weil homomorphism}
\begin{equation}\label{equation;equivariant-characteristic}
c_{\h,\theta}\colon(S\liek^*)^\liek\longto(\A_\bbas{\liek})_\h.
\end{equation}
See~\cite[Theorem~A.6.3]{lin-sjamaar;thom} for the next result.

\begin{theorem}\label{theorem;equivariant-ccw}
Let $\g=\liek\times\h$ be the product of two Lie algebras $\liek$ and
$\h$.  Let $\A$ be a commutative $\g$-differential graded algebra.
Suppose that $\A$ is $\liek$-locally free and admits an $\h$-invariant
$\liek$-connection $\theta\colon\liek^*\to\A^1$.  The morphism
$j\colon\W\h\otimes\A\to\W\g\otimes\A$ given by the inclusion
$\W\h\to\W\g$ is a homotopy equivalence with homotopy inverse
\[
\begin{tikzcd}
\W\g\otimes\A\cong\W\liek\otimes\W\h\otimes\A\ar[r,"\simeq"]&
\W\h\otimes\A
\end{tikzcd}
\]
defined by $x\otimes y\otimes a\mapsto c_\theta(x)y\otimes a$, where
$c_\theta$ is as in~\eqref{equation;equivariant-ccw}.  Hence $j$
induces a homotopy equivalence
\[
\begin{tikzcd}
(\W\h\otimes\A)_\bbas{\g}=(\A_\bbas{\liek})_\h\ar[r,"\simeq"]&
  (\W\g\otimes\A)_\bbas{\g}=\A_\g
\end{tikzcd}
\]
and an isomorphism
$\begin{tikzcd}[cramped,sep=small]H_\h(\A_\bbas{\liek})\ar[r,"\cong"]&
  H_\g(\A)\end{tikzcd}$.
\end{theorem}

\subsection{Equivariant basic differential forms}
\label{section;basic-equivariant}

Let $(M,\F)$ be a foliated manifold, $\g$ a finite-dimensional Lie
algebra, and $a\colon\g\to\X(M,\F)$ a transverse $\g$-action on $M$.
For $\xi\in\g$ let $\xi_M=a(\xi)\in\X(M,\F)$ denote the transverse
vector field on $M$ defined by the $\g$-action.  For
$\alpha\in\Omega(M,\F)$ define
\[
\iota(\xi)\alpha=\iota(\tilde{\xi}_M)\alpha,\qquad
L(\xi)\alpha=L(\tilde{\xi}_M)\alpha,
\]
where $\tilde{\xi}_M$ is a foliate vector field that represents
$\xi_M$.  Since $\alpha$ is $\F$-basic, these contractions and
derivatives are independent of the choice of the representative
$\tilde{\xi}_M$ of $\xi_M$.  Goertsches and
T\"oben~\cite[Proposition~3.12]%
{goertsches-toeben;equivariant-basic-riemannian} observed that,
equipped with these operations, the basic de Rham complex
$\Omega(M,\F)$ a $\g$-differential graded algebra, and they called
elements of the Weil complex
\[
\Omega_\g(M,\F)=(\W\g\otimes\Omega(M,\F))_\bbas\g
\]
\emph{$\g$-equivariant $\F$-basic differential forms}.  The cohomology
\[H_\g(M,\F)=H(\Omega_\g(M,\F))\]
is the \emph{$\g$-equivariant $\F$-basic de Rham cohomology} of the
foliated manifold with respect to the transverse action.  We will
usually say ``equivariant basic'' instead of ``$\g$-equivariant
$\F$-basic''.
\glossary{O.megaMFg@$\Omega_\g(M,\F)$, equivariant basic de Rham
  complex}

Let $(M',\F')$ be another foliated manifold equipped with a transverse
$\g$-action.  We say that a foliate map $f\colon M\to M'$ is
\emph{$\g$-equivariant} if the transverse vector fields $\xi_M$ and
$\xi_{M'}$ are $f$-related as in~\eqref{equation;related}.  A
$\g$-equivari\-ant foliate map $f$ induces a morphism of
$\g$-differential graded algebras $\Omega(M',\F')\to\Omega(M,\F)$.  We
call a smooth map $f\colon[0,1]\times M\to M'$ a
\emph{$\g$-equivariant foliate homotopy} if the map $f_t\colon M\to
M'$ defined by $f_t(x)=f(t,x)$ is $\g$-equivariant foliate for all
$t\in[0,1]$.  We quote the following equivariant basic homotopy lemma
from~\cite[Lemma~4.2.1]{lin-sjamaar;thom}.

\begin{lemma}\label{lemma;cochain-homotopy}
Let $(M,\F)$ and $(M',\F')$ be foliated manifolds equipped with
transverse actions of a Lie algebra $\g$.  Let $f\colon[0,1]\times
M\to M'$ be a $\g$-equivariant foliate homotopy.  Then the pullback
morphisms $f_0$ and $f_1\colon\Omega(M',\F')\to\Omega(M,\F)$ are
homotopic as morphisms of $\g$-differential graded algebras.  In
particular they induce the same homorphisms in equivariant basic
cohomology: $f_0^*=f_1^*\colon H_\g(M',\F')\to H_\g(M,\F)$.
\end{lemma}

\section{Molino theory and some applications}\label{section;molino}

This section is a pr\'ecis of Molino's structure theory of Riemannian
foliations, partly based
on~\cite[\S\,4]{goertsches-toeben;equivariant-basic-riemannian}.  A
full account of these results can be found
in~\cite{molino;riemannian-foliations} and also
in~\cite[Ch.\ 4]{moerdijk-mrcun;foliations-groupoids}.  We restate
them slightly in terms of transverse action Lie algebroids, and to
Molino's two structure theorems we add a third structure theorem,
which concerns a dual pair of bundles of Lie algebras and a
generalized morphism of groupoids associated to a Riemannian
foliation.  At the end we state several corollaries that will be used
in later sections.

A \emph{transverse Riemannian metric} on a foliated manifold $(M,\F)$
is a positive definite symmetric bilinear form $g$ on the normal
bundle $N\F$ with the property that $L(v)g=0$ for all vector fields
$v\in\X(\F)$ tangent to the leaves.  The pair $(\F,g)$ is then called
a \emph{Riemannian foliation}.  A \emph{bundle-like Riemannian metric}
on $M$ is a Riemannian metric $g_{TM}$ on $M$ with the property that
the function $g(v,w)$ is basic for all foliate vector fields $v$ and
$w\in\lie{N}(\F)$ that are perpendicular to the leaves of $\F$.  A
bundle-like metric $g$ gives rise to a transverse metric $g$ by
identifying $N\F$ with the $g_{TM}$-orthogonal complement of $\F$ and
then restricting $g_{TM}$ to $N\F$.  Conversely, for every transverse
metric $g$ there is a bundle-like metric $g_{TM}$ which induces $g$;
see~\cite[\S\,3.2]{molino;riemannian-foliations}.
Following~\cite[Definition~3.1]%
{goertsches-nozawa-toeben;chern-simons-foliations}
or~\cite[\S\,2.1]{goertsches-toeben;equivariant-basic-riemannian} we
call a Riemannian foliation $(\F,g)$ on $M$ \emph{(metrically)
  complete} if there exists a bundle-like metric $g_{TM}$ which
induces $g$ and which is complete.  A foliate map $f$ from $M$ to a
second manifold $M'$ equipped with a Riemannian foliation $(\F',g')$
is a \emph{transverse Riemannian submersion} if it is a submersion and
if the normal derivative $N_xf\colon N_x\F\to N_{f(x)}\F'$ preserves
scalar products of vectors perpendicular to $\ker(N_xf)$.
\glossary{MFg@$(M,\F,g)$, Riemannian foliated manifold}
%

\subsection{Notation and conventions}\label{section;notation-molino}

In the rest of this section $M$ denotes a connected manifold equipped
with a metrically complete Riemannian foliation $(\F,g)$.  We denote
the (constant) codimension $\dim(M)-\dim(\F(x))$ of the foliation by
$q$, the orthogonal group $\group{O}(q)$ by $K$, and its Lie algebra
by $\liek$.  We define $g_K$ to be the bi-invariant Riemannian metric
on $K$ with respect to which $K$ has volume $1$ and we let $g_{\liek}$
be the associated inner product on $\liek$.  We denote by $V_M$ the
trivial vector bundle over $M$ with fibre a vector space $V$.  See
\S\,\ref{section;foliation} and the notation index in the back for
general notational conventions.
\glossary{gK@$g_K$, normalized bi-invariant Riemannian metric on
  $K=\group{O}(q)$}
\glossary{gk@$g_{\liek}$, inner product on $\liek=\lie{o}(q)$ induced
  by $g_K$}
%

\subsection{The Molino diagram}\label{section;molino-diagram}

Define a relation among points of $M$ as follows: $x\sim y$ if the
closures of the leaves $\F(x)$ and $\F(y)$ intersect.  The gist of
Molino's theory is that, unlike for general foliations, this is an
equivalence relation, and that $x\sim y$ if and only if the closure of
the leaf $\F(x)$ is \emph{equal} to the closure of the leaf $\F(y)$.
The leaf closures are the leaves of a singular foliation $\barF$, we
can form the \emph{leaf closure space} $M/\barF$, which is a Hausdorff
topological space, and we have a continuous map $M/\F\to M/\barF$.
The kicker is that this leaf closure space is the quotient space of a
Riemannian manifold $W$, which we call the \emph{Molino manifold} and
whose definition we review below, by an isometric action of the
orthogonal group $K=\group{O}(q)$.  This enables us to reduce certain
questions about the foliation $\F$ to usually much easier questions
about the $K$-action on $W$.  The situation can be summarized by the
\emph{Molino diagram},
\begin{equation}\label{equation;molino}
\begin{tikzcd}[column sep=-2em]
&&(P,\F_P,g_P)\ar[dll,"\pi"']\ar[drr,"\varrho"]&&
\\
(M,\F,g)\ar[dr,"/\F"'near start]&&&&(W,g_W)\ar[dl,"/K"near start]
\\
&M/\F\ar[rr]&&M/\barF&
\end{tikzcd}
\end{equation}
\glossary{Fbar@$\barF$, closure of Riemannian foliation}
\glossary{P@$P$, Molino bibundle}
\glossary{W@$W$, Molino manifold}

The manifold $P$, which we call the \emph{Molino bibundle}, is the
bundle of orthonormal frames of the normal bundle $N\F$, and is a
foliated principal $K$-bundle over $M$ with structure group
$K=\group{O}(q)$.  (In case $\F$ is transversely orientable, we choose
a transverse orientation, we take $P$ to consist of \emph{oriented}
orthonormal frames, and we replace $K$ by $\SO(q)$.)  We describe the
foliation $\F_P$ by specifying a partial connection on $P$ as follows.
A \emph{Killing vector field} on $(M,\F)$ is a vector field
$v\in\X(M)$ that satisfies $L(v)g=0$ (not $L(v)g_{TM}=0$!).
By~\cite[Lemma 3.5]{molino;riemannian-foliations} all Killing vector
fields are foliate and therefore they form a Lie subalgebra of
$\lie{N}(\F)$, which we denote by $\lie{N}(\F,g)$.  Vector fields in
$\X(\F)$ are by definition Killing, so $\X(\F)$ is an ideal of
$\lie{N}(\F,g)$.  We denote the quotient Lie algebra, which is a
subalgebra of $\X(M,\F)$, by $\X(M,\F,g)$ and call its elements
\emph{transverse Killing vector fields}.  The short exact
sequence~\eqref{equation;foliate-transverse} restricts to a short
exact sequence
\[
\X(\F)\longinj\lie{N}(\F,g)\longsur\X(M,\F,g).
\]
The flow $\phi_t$ of a Killing vector field $v$ is a $1$-parameter
group of foliate diffeomorphisms, so the normal derivative $N\phi_t$
is a $1$-parameter group of vector bundle automorphisms of $N\F$.  The
flow $N\phi_t$ preserves the metric $g$, so it maps orthonormal frames
to orthonormal frames and therefore lifts naturally to a
$K$-equivariant flow $\phi_{P,t}$ of bundle automorphisms of $P$.  The
infinitesimal generator $\pi^\dag(v)$ of $\phi_{P,t}$ is a
$K$-invariant vector field, and the map $v\mapsto\pi^\dag(v)$ is a
homomorphism of Lie algebras
\begin{equation}\label{equation;killing}
\pi^\dag\colon\lie{N}(\F,g)\longto\X(P)^K,
\end{equation}
which we call the \emph{natural lifting homomorphism}.  The
restriction of $\pi^\dag$ to $\X(\F)$ gives a vector bundle map
$\pi^*T\F\to TP$.  This is the partial connection on $P$.  Its image
is the tangent bundle to the foliation $\F_P$ that makes $P$ a
foliated principal $K$-bundle.
\glossary{NNFg@$\lie{N}(\F,g)$, Killing vector fields on $(M,\F,g)$}
\glossary{XMFg/F@$\X(M,\F,g)$, transverse Killing vector fields on
  $(M,\F,g)$}
\glossary{p.i@$\pi$, Molino bibundle projection}
\glossary{r.ho@$\varrho$, Molino bibundle projection}
\glossary{p.idagger@$\pi^\dag$, natural lifting homomorphism}
\glossary{r.hodagger@$\varrho^\dag$, natural descent homomorphism}

The partial connection extends to a unique torsion-free connection,
the \emph{transverse Levi-Civita connection}, which is $\F_P$-basic;
see~\cite[\S\,3.3]{molino;riemannian-foliations}.  We denote this
connection, and its associated connection $1$-form in
$\Omega^1(P,\F_P;\liek)^K$, by $\theta_\LC$.  The
$\theta_\LC$-horizontal lift of a vector field $v\in\X(M)$ agrees with
the canonical lift $\pi^\dag(v)$ if $v$ is tangent to $\F$, but may
differ from $\pi^\dag(v)$ if $v$ is an arbitrary Killing vector field.
Since $\theta_\LC$ maps $\pi^*T\F$ to $T\F_P$, it induces a vector
bundle map $\bar{\theta}_\LC\colon\pi^*N\F\to N\F_P$, which is a
splitting (right inverse) of the normal derivative $N\pi\colon
N\F_P\to\pi^*N\F$ of $\pi$.
\glossary{t.hetaLC@$\theta_\LC$, transverse Levi-Civita connection}

The \emph{solder form} or \emph{fundamental form} of $P$ is the
$\R^q$-valued $1$-form $\sigma\in\Omega^1(P,\F_P;\R^q)^K$ defined as
follows: let $p\in P$, $v\in T_pP$, and $x=\pi(p)$, view the
orthonormal frame $p$ as an isometry $p\colon\R^q\to N_x\F$, and put
\[\sigma_p(v)=p^{-1}\bigl(T_p\pi(v)\bmod T_x\F\bigr)\in\R^q.\]    
The solder form is equivariant with respect to the $K$-actions on $P$
and $\R^q$.  For every $p$ the linear map $\sigma_p\colon T_pP\to\R^q$
is surjective and its kernel is equal to the sum $T_p(K{\cdot}p)\oplus
T_p\F_P$.  Therefore $\sigma$ descends to a surjective vector bundle
map $\bar\sigma\colon N\F_P\to\R^q_P$ with kernel $\liek_P$.  The
connection form $\theta_\LC$ and the solder form $\sigma$ give us the
\emph{soldering diagram} of the foliated bundle $P$,
\begin{equation}\label{equation;solder}
\begin{tikzcd}
&\liek_P\ar[r,"="]\ar[d,hook]&\liek_P\ar[d,hook]&
\\
T\F_P\ar[r,hook]\ar[d,hook,two heads]& TP\ar[r,two heads]\ar[d,two
  heads,"T\pi"]& N\F_P\ar[d,two heads,"N\pi"]\ar[dr,two
  heads,"\bar\sigma"]&
\\
\pi^*T\F\ar[r,hook]&\pi^*TM\ar[r,two heads]\ar[u,bend
  left,"\theta_\LC"]&\pi^*N\F\ar[u,bend left,"\bar\theta_\LC"]&\R^q_P
\end{tikzcd}
\end{equation}
The diagram commutes, its rows and columns are exact, and $\theta_\LC$
and $\bar\theta_\LC$ are splittings, so we get isomorphisms
$\bar\sigma\circ\bar\theta_\LC\colon\pi^*N\F\cong\R^q_P$ and
\begin{equation}\label{equation;solder-split}
N\F_P\cong\liek_P\oplus\R^q_P.
\end{equation}
Thus a choice of bases of $\liek$ and $\R^q$ gives rise to a
\emph{transverse parallelism} of $P$, i.e.\ a global frame of the
normal bundle $N\F_P$ given by sections that are constant with respect
to the trivialization~\eqref{equation;solder-split}.  Constant
sections of $N\F_P$ are transverse vector fields on $P$.
\glossary{s.igma@$\sigma$, solder form}

The following theorem, which is an excerpt of Molino's results, is a
list of the main features of the Molino
diagram~\eqref{equation;molino}.  The notation is as in
\S\,\ref{section;notation-molino} and as in the
diagrams~\eqref{equation;molino} and~\eqref{equation;solder}.

\begin{theorem}[first structure theorem]\label{theorem;molino-1}
Let $M$ be a connected manifold equipped with a complete Riemannian
foliation $(\F,g)$.
\begin{enumerate}
\item\label{item;transverse-metric}
There is a unique transverse Riemannian metric $g_P$ on the Molino
bibundle $(P,\F_P)$ with respect to which the
trivialization~\eqref{equation;solder-split} is an isometry.  This
transverse metric $g_P$ is $K$-invariant and complete, and the
projection $\pi\colon P\to M$ is a transverse Riemannian submersion.
The natural lift of a Killing vector field on $(M,\F,g)$ is a
$K$-invariant Killing vector field on $(P,\F_P,g_P)$, so the natural
lifting homomorphism $\pi^\dag$ defined in~\eqref{equation;killing}
descends to a Lie algebra homomorphism
\[
\pi^\dag\colon\X(M,\F,g)\longto\X(P,\F_P,g_P)^K.
\]
\item\label{item;homogeneous}
The foliation $\F_P$ is transversely parallelizable and homogeneous.
Hence the leaf closure foliation $\barF_P$ is strictly simple and the
space of leaf closures is a manifold $P/\barF_P=W$.  The $K$-action on
$(P,\F_P)$ is foliate and therefore descends uniquely to a smooth
$K$-action on $W$.  The quotient map $\varrho\colon P\to W$ is a
$K$-equivariant locally trivial fibre bundle.  There exists a unique
Riemannian metric $g_W$ on $W$ with respect to which $\varrho$ is a
transverse Riemannian submersion.  The metric $g_W$ is complete.
Every transverse vector field $v$ on $P$ is $\varrho$-related to a
unique vector field $\varrho_\dag(v)$ on $W$.  The map
\[
\varrho_\dag\colon\X(P,\F_P)\longto\X(W)
\]
is a Lie algebra homomorphism.  If $v\in\X(P,\F_P)$ is Killing, then
so is $\varrho_\dag(v)$.
\item\label{item;leaf-closure}
Let $L$ be a leaf of $\F$ and let $L_P$ be a leaf of $\F_P$ mapping to
$L$.  Then $\bar{L}=\pi(\bar{L}_P)$ and $\bar{L}$ is an embedded
submanifold of $M$.  The $K$-equivariant map $\varrho\colon P\to W$
induces a continuous map $M=P/K\to W/K$, which descends to a
homeomorphism $M/\barF\to W/K$.
\end{enumerate}
\end{theorem}

Molino also showed that each leaf closure $\barF(x)$ of $M$ is the
orbit of the leaf $\F(x)$ under the action of a certain flat bundle of
Lie algebras over $M$.  He did not point out, but it follows
immediately from his results, that $\barF(x)$ is actually the orbit of
$x$ under the action of a certain Lie algebroid.  We reformulate
Molino's second structural theorem in terms of this Lie algebroid as
follows.  The \emph{transverse tangent sheaf} $\ca{T}_M=\ca{T}$ of
$(M,\F)$ is the sheaf of Lie algebras on $M$ associated to the
presheaf $U\mapsto\X(U,\F|_U)$, where $U$ ranges over all open subsets
of $M$.  In other words, $\ca{T}$ is the quotient of the sheaves
$U\mapsto\lie{N}(\F|_U)$ and $U\mapsto\X(\F|_U)$.  The
\emph{transverse Killing sheaf} $\ca{K}_M=\ca{K}$ is the subsheaf of
$\ca{T}$ associated to the presheaf $U\mapsto\X(U,\F|_U,g)$.  We have
a similar sheaves $\ca{T}_P$ and $\ca{K}_P$ on $P$.  The Lie algebra
of global transverse vector fields $\X(P,\F_P)$ is a subspace of the
space of global sections of $\ca{T}_P$.  The sheaf $\ca{T}_P$ is
constant, because $P$ has a transverse parallelism.  Molino's
\emph{centralizer sheaf} of $P$ is the subsheaf $\ca{C}_P$ of
$\ca{T}_P$ consisting of all sections that centralize the global
transverse vector fields:
\[
\ca{C}_P(V)=\{\,v\in\ca{T}_P(V)\mid\text{$[v,w]=0$ for all
  $w\in\X(P,\F_P)$}\,\}.
\]
The natural lifting homomorphism provides a homomorphism of sheaves of
Lie algebras $\pi^\dag\colon\ca{K}\to\pi_*\ca{T}_P$.  The
\emph{centralizer sheaf} $\ca{C}=\ca{C}_M$ of $M$ is the inverse image
of $\ca{C}_P$ under this homomorphism.  In other words, $\ca{C}$ is
the subsheaf of $\ca{K}$ consisting of all sections $u\in\ca{K}(U)$
with the property that $\pi^\dag(u)\in\ca{C}_P(\pi^{-1}(U))$.
\glossary{Fxbar@$\barF(x)$, leaf closure of $x$}
\glossary{C@$\ca{C}$, centralizer sheaf of $M$}

The next result, which again is excerpted from Molino's
book~\cite{molino;riemannian-foliations}, explains that the closures
of the leaves of $\F$ are the orbits of the sheaf of Lie algebras
$\ca{C}$, and therefore constitute a singular foliation $\barF$ of
$M$.  The notation is as in Theorem~\ref{theorem;molino-1}.  See also
the table at the end of the notation index.

\begin{theorem}[second structure theorem]\label{theorem;molino-2}
Let $M$ be a connected manifold equipped with a complete Riemannian
foliation $(\F,g)$.
\begin{enumerate}
\item\label{item;stalk}
The sheaves $\ca{C}_P$ and $\ca{C}$ are locally constant.  The natural
lifting homomorphism $\pi^\dag$ induces an isomorphism
$\ca{C}\cong(\pi_*\ca{C}_P)^K$.  Let $p\in P$ and let $x=\pi(p)\in M$.
Let $\Lambda$ be the closure of the leaf $\F_P(p)$, let $\F_\Lambda$
be the restriction of the foliation $\F_P$ to $\Lambda$, and let
$\ca{T}_\Lambda$ be the transverse tangent sheaf of
$(\Lambda,\F_\Lambda)$.  The stalk $\ca{C}_{P,p}$ is a Lie algebra
equal to the centralizer of $\X(\Lambda,\F_\Lambda)$ in
$\ca{T}_{\Lambda,p}$.  The dimension of $\ca{C}_{P,p}$ is
\[\dim(\ca{C}_{P,p})=\dim\bigl(\barF_P(p)\bigr)-\dim(\F_P(p)).\]
\item\label{item;local-system}
The sheaf $\ca{C}_P$ is the sheaf of flat sections of a
$K$-equivariant flat bundle of Lie algebras $\liec_P$, whose fibre at
$p$ is equal to the stalk $\ca{C}_{P,p}$.  The evaluation maps
$\ca{C}_{P,p}\to N_p\F_P$ give rise to a bundle map $\liec_P\to
N\F_P$.  Similarly, the sheaf $\ca{C}_M=\ca{C}$ is the sheaf of flat
sections of a flat bundle of Lie algebras $\liec_M=\liec$, whose fibre
at $x$ is equal to the stalk $\ca{C}_x$; and $\liec$ is equipped with
a bundle map $\liec\to N\F$.  We have isomorphisms
$\pi^*\liec\cong\liec_P$ and $\liec=\liec_P/K$.
\item\label{item;algebroid}
The fibred product
\[\liec_P\ltimes\F_P=\liec_P\times_{N\F_P}TP\]
is a Lie algebroid over $P$, which acts freely on $P$ and whose orbits
are the closures of the leaves of $\F_P$.  The decomposition into leaf
closures $\barF_P$ is a foliation of $P$ of (constant) codimension
$q+\frac12q(q-1)-\dim(\ca{C}_{P,p})$.  Similarly,
\[\liec\ltimes\F=\liec\times_{N\F}TM\]
is a Lie algebroid over $M$, whose orbits are the closures of the
leaves of $\F$.  The decomposition into leaf closures $\barF$ is a
singular foliation of $M$.  The codimension of $\barF(x)$ is
$q+\dim(\Stab(K,w))-\dim(\ca{C}_{P,p})$, where $x=\pi(p)$ and
$\Stab(K,w)$ denotes the stabilizer of $w=\varrho(p)\in W$ under the
$K$-action.
\end{enumerate}
\end{theorem}
\glossary{Stab@$\Stab(G,x)$, stabilizer group of $x$ with respect to
  Lie group or Lie groupoid $G$}
\glossary{c@$\liec$, centralizer bundle of $M$}
\glossary{cF@$\liec\ltimes\F$, centralizer Lie algebroid of $M$}

We call $\liec_P$ and $\liec$ the \emph{centralizer bundles} of $P$,
resp.\ $M$.  We call $\liec_P\ltimes\F_P$ and $\liec\ltimes\F$ the
\emph{centralizer Lie algebroids} of $P$, resp.\ $M$.

A transverse parallelism of $P$ is the same thing as a basis of the
space of transverse vector fields $\X(P,\F_P)$ considered as a module
over the ring of basic functions $\Omega^0(P,\F_P)\cong C^\infty(W)$.
In particular, this module is free over $C^\infty(W)$ (of rank
$q+\frac12q(q-1)$) and is therefore the module of sections of a
(trivial) vector bundle $\lie{b}$ over $W$.  The Lie bracket on
$\X(P/\F_P)$ and the surjective homomorphism
$\varrho_\dag\colon\X(P/\F_P)\to\X(W)$ make $\lie{b}$ a transitive Lie
algebroid over $W$, called the \emph{basic Lie algebroid} in
\cite[\S\,6.4]{moerdijk-mrcun;foliations-groupoids}.  The anchor of
$\lie{b}$ being surjective, the stabilizers
$\lie{s}_{W,w}=\stab(\lie{b},w)$ for $w\in W$ form a locally trivial
Lie algebra bundle $\lie{s}_W$ (which is denoted by $\g_W$
on~\cite[p.~119]{molino;riemannian-foliations}).

The first three items of the next result are a reformulation of
further results of Molino.  The last item, which is not in Molino, but
which we state without proof as we won't use it here, justifies our
usage of the term ``bibundle'' for the transverse frame bundle $P$,
and it places Molino's structure theory in the context of Lie
groupoids.  See~\cite[\S\,3.2]{lerman;orbifolds-as-stacks} for an
explanation of bibundles.

\begin{theorem}[third structure theorem]\label{theorem;molino-3}
Let $M$ be a connected manifold equipped with a complete Riemannian
foliation $(\F,g)$.
\begin{enumerate}
\item
%
Let $w\in W$.  Let $\Lambda_w=\varrho^{-1}(w)$ and let $\F_w$ be the
restriction of the foliation $\F_P$ to $\Lambda_w$.  Then
$\lie{s}_{W,w}$ is isomorphic to the finite-dimensional Lie algebra
$\X(\Lambda_w/\F_w)$.  The Lie algebra bundles
$\lie{s}_P=\varrho^*\lie{s}_W$ and $\liec_P$ over $P$ form a dual pair
in the sense that for each $p\in P$ the fibres $\lie{s}_{P,p}$ and
$\liec_{P,p}$ are each other's centralizer in the Lie algebra
$\ca{T}_{P,p}$.
\item\label{item;lie-group-bundle}
The Lie algebra bundle $\lie{s}_W$ integrates to a locally trivial Lie
group bundle $S_W$ with simply connected fibres.  For $w\in W$ let
$\tilde{P}_w$ be the Darboux cover of the leaf closure
$\varrho^{-1}(w)$.  The union $\tilde{P}=\bigcup_{w\in W}\tilde{P}_w$
is a fibre bundle over $W$ equipped with a bundle map $\tilde{P}\to
S_W$.  For each $w\in W$ the fibre $\lie{s}_{W,w}$ of $\lie{s}_W$ is
the Lie algebra of left-invariant vector fields on $S_w$ and for every
$p\in\varrho^{-1}(w)$ the fibre $\liec_{P,p}$ of $\liec_P$ is the Lie
algebra of right-invariant vector fields on $S_w$.
\item\label{item;molino-morita}
The manifold $P$ is a bibundle for the holonomy groupoid $\Hol(M,\F)$
and the action groupoid $K\ltimes W$.  Therefore $P$ defines a
generalized groupoid morphism from $\Hol(M,\F)$ to $K\ltimes W$, as
well as a morphism from the leaf stack $[M/\F]$ to the quotient stack
$[W/K]$.
\end{enumerate}
\end{theorem}

\subsection{Some corollaries of Molino theory}
\label{section;molino-corollary}

The Molino structure theorems,
Theorems~\ref{theorem;molino-1}--\ref{theorem;molino-3}, have useful
things to say about transverse Lie algebra actions.  A transverse
action $a\colon\g\to\X(M,\F)$ of a finite-dimensional Lie algebra $\g$
on a Riemannian foliated manifold $(M,\F,g)$ is \emph{isometric} if
the homomorphism $a$ takes values in the transverse Killing vector
fields $\X(M,\F,g)$.  For the remainder of
\S\,\ref{section;molino-corollary} we fix an isometric transverse
$\g$-action $a$ on $M$.  We denote by $\X(W,g_W)$ the Lie algebra of
Killing vector fields of the Molino manifold $(W,g_W)$.
\glossary{XWg@$\X(W,g_W)$, Killing vector fields on Molino manifold
  $(W,g_W)$}

Our first application of Molino theory is that isometric transverse
Lie algebra actions correspond to isometric Lie \emph{group} actions
on the Molino manifold.

\begin{proposition}\label{proposition;induced-action}
Let $M$ be a connected manifold equipped with a complete Riemannian
foliation $(\F,g)$ and let $a\colon\g\to\X(M,\F,g)$ be an isometric
transverse action.  The homomorphism $a_P=\pi^\dag\circ
a\colon\g\to\X(P,\F_P,g_P)$ is an isometric transverse $\g$-action on
$P$.  The homomorphism $a_W=\varrho_\dag\circ a_P\colon\g\to\X(W,g_W)$
is an isometric $\g$-action on $W$.  The bibundle projections
$\pi\colon P\to M$ and $\varrho\colon P\to W$ are equivariant with
respect to these actions.  Let $G$ be a simply connected Lie group
with Lie algebra $\g$.  The action $a_W$ integrates to an isometric
$G$-action, which commutes with the $K$-action.
\end{proposition}

\begin{proof}
The first two assertions follow immediately from
Theorem~\ref{theorem;molino-1}\eqref{item;transverse-metric}--%
\eqref{item;homogeneous}.  Also by Theorem~\ref{theorem;molino-1}, our
hypothesis that $(M,\F,g)$ is complete implies that the Riemannian
manifold $(W,g_W)$ is complete, so the Killing vector field $\xi_W$ is
complete for each $\xi\in\g$, and therefore the Lie algebra action
$a_W$ integrates to a $G$-action by the Lie-Palais theorem.  The
projection $\pi\colon P\to M$ is $K$-invariant, so the transverse
vector fields $\xi_P$ commute with the vector fields $\eta_P$ for all
$\xi\in\g$ and $\eta\in\liek$, and so the $G$-action on $W$ commutes
with the $K$-action.
\end{proof}

It follows from Lemma~\ref{lemma;stabilizer}\eqref{item;same} that if
$y$ is in the leaf closure of a point $x$, then we have an inclusion
of stabilizers $\stab(\g\ltimes\F,x)\subseteq\stab(\g\ltimes\F,y)$.
The next observation shows that in the situation of Riemannian
foliations this inclusion is an equality.

\begin{proposition}\label{proposition;closure-stabilizer}
Let $M$ be a connected manifold equipped with a complete Riemannian
foliation $(\F,g)$ and let $a\colon\g\to\X(M,\F,g)$ be an isometric
transverse action.
\begin{enumerate}
\item\label{item;commute}
The transverse $\g$-action $a$ commutes with the transverse action of
the centralizer Lie algebroid $\liec$ defined in
Theorem~\ref{theorem;molino-2}.
\item\label{item;closure-stabilizer}
Let $x$ and $y\in M$.  If $y\in\barF(x)$, then
$\stab(\g\ltimes\F,x)=\stab(\g\ltimes\F,y)$.
\item\label{item;product-lie-algebroid}
The transverse action Lie algebroid $(\g\times\liec)\ltimes\F$ is
well-defined.  Its orbit through a point $x\in M$ is equal to
\[
(\g\times\liec)\ltimes\F(x)=\bigcup_{y\in\barF(x)}\g\ltimes\F(y)=
\bigcup_{z\in\\g\ltimes\F(x)}\barF(z).
\]
\item\label{item;locally-closed}
Every locally closed $\g\ltimes\F$-invariant subset of $M$ is
$(\g\times\liec)\ltimes\F$-invariant.
\end{enumerate}
\end{proposition}

\begin{proof}
\eqref{item;commute}~Let $\xi\in\g$.  By definition the transverse
vector field $\pi^\dag(\xi_M)$ commutes with the sections of the
centralizer subsheaf $\ca{C}_P$.  Since the lifting homomorphism
$\pi^\dag$ induces an isomorphism $\ca{C}\cong\pi_*\ca{C}_P$
(Theorem~\ref{theorem;molino-2}\eqref{item;stalk}), it follows that
$\xi_M$ commutes with the sections of $\ca{C}$.

\eqref{item;closure-stabilizer}~The leaf closure $\barF(x)$ is the
orbit of $x$ under the centralizer Lie algebroid
(Theorem~\ref{theorem;molino-2}\eqref{item;algebroid}).
By~\eqref{item;commute} the action of this Lie algebroid commutes with
that of $\g$, so $y$ has the same stabilizer as $x$.

\eqref{item;product-lie-algebroid}~It follows
from~\eqref{item;commute} that the bundle of Lie algebras
$\g\times\liec$ acts transversely on $M$, so we have a well-defined
transverse action Lie algebroid $(\g\times\liec)\ltimes\F$.  The orbit
$(\g\times\liec)\ltimes\F(x)$ can be alternatively described as the
$\g$-orbit of the leaf closure $\barF(x)=\liec(x)$ or as the
$\liec$-orbit of the $\g$-orbit of the leaf $\F(x)$.

\eqref{item;locally-closed}~It follows
from~\eqref{item;product-lie-algebroid} that every closed
$\g\ltimes\F$-invariant subset of $M$ is
$(\g\times\liec)\ltimes\F$-invariant.  Now let $X\subseteq M$ be
locally closed and $\g\ltimes\F$-invariant.  Then $\bar{X}$ and
$\bar{X}\backslash X$ are closed and $\g\ltimes\F$-invariant, hence
$(\g\times\liec)\ltimes\F$-invariant, so $X$ itself is
$(\g\times\liec)\ltimes\F$-invariant.
\end{proof}

\begin{remark}\label{remark;invariant}
Item~\eqref{item;closure-stabilizer} of this proposition is true even
if the transverse $\g$-action is not isometric.  The reason is that a
foliate vector field on $M$ that is tangent to $\F(x)$ for some $x\in
M$ is tangent to $\F(y)$ for every $y$ in the leaf closure of $x$.
(Cf.\ remark on~\cite[p.~325]{belfi-park-richardson;hopf}.)
Item~\eqref{item;locally-closed} shows that every
$\g\ltimes\F$-invariant embedded submanifold of $M$ is
$(\g\times\liec)\ltimes\F$-invariant.
\end{remark}

Another desirable item is the existence of suitable tubular
neighbourhoods.  We call a subset $X$ of $M$ \emph{invariant} or
\emph{saturated} with respect to a Lie algebroid $\liea$ over $M$ if
for every $x\in X$ the $\liea$-orbit of $x$ is contained in $X$.  Let
$X$ be a $\g\ltimes\F$-invariant embedded submanifold of $M$ and let
$\F_X=\F|_X$ be the restriction of the foliation to $X$.  Then the
normal bundle
\[N_MX=NX=TM|_X/TX\cong N\F/N\F_X\]
is a foliated vector bundle over $(X,\F_X)$ and is equipped with a
natural transverse $\g$-action with the property that the bundle
projection $NX\to X$ is equivariant.  A \emph{$\g\ltimes\F$-invariant
  tubular neighbourhood} of $X$ is a $\g$-equivariant foliate
embedding $f\colon NX\to M$ with the following properties: the image
$f(NX)$ is a $\g\ltimes\F$-invariant open subset of $M$; $f|_X=\id_X$;
and $T_xf=\id_{N_xX}$ for all $x\in X$.  (Here we identify $X$ with
the zero section of $NX$ and the normal bundle of $X$ in $NX$ with
$NX$.)
\glossary{NMX@$N_MX$, normal bundle of submanifold $X$}

We cannot guarantee that every $\g\ltimes\F$-invariant $X$ has a
$\g\ltimes\F$-invariant tubular neighbourhood: we must require in
addition that $X$ be closed.  (By Remark~\eqref{remark;invariant} $X$
and its tubular neighbourhood are then automatically invariant under
the bigger Lie algebroid $(\g\times\liec)\ltimes\F$.)
\glossary{X@$(X,\F_X)$, $\F$-invariant submanifold of $M$ with induced
  foliation}
\glossary{FX@$\F_X$, induced foliation of $X$}
\glossary{X@$X_W$, submanifold of $W$ corresponding to $X$}

\begin{proposition}\label{proposition;tubular}
Let $M$ be a connected manifold equipped with a complete Riemannian
foliation $(\F,g)$ and let $a\colon\g\to\X(M,\F,g)$ be an isometric
transverse action.  Every $\g\ltimes\F$-invariant closed embedded
submanifold $X$ of $M$ has a $\g\ltimes\F$-invariant tubular
neighbourhood.  For every pair of $\g\ltimes\F$-invariant tubular
neighbourhoods $f_0$, $f_1\colon NX\to M$ there exists a
$\g$-equivariant foliate isotopy from $f_0$ to~$f_1$.
\end{proposition}

\begin{proof}
Let $X_P=\pi^{-1}(X)$ and $X_W=\varrho(X_P)$.  Then $X_P$ is a
$K$-invariant and $\g\ltimes\F_P$-invariant closed embedded
submanifold of $P$, and $X_W$ is a $G\times K$-invariant closed
embedded submanifold of $W$, where $G$ is as in
Proposition~\ref{proposition;induced-action}.  The $G\times K$-action
on $W$ preserves the metric $g_W$.  Let $H$ be the closure of the
image of $G\times K$ in the isometry group of $W$.  Then $X_W$, being
closed, is $H$-invariant, and $H$ acts properly on $W$, so the
$H$-equivariant version of the standard tubular neighbourhood theorem
holds.  Let $f_W\colon NX_W\to W$ be a $G\times K$-invariant tubular
neighbourhood of $X_W$.  Pulling back through $\varrho$ we obtain an
embedding $f_P\colon\varrho^*NX_W\cong NX_P\to P$, which is a
$\g$-equivariant and $K$-invariant tubular neighbourhood of $X_P$.
Hence the quotient by $K$ is an embedding $f\colon NX\to M$, which is
a $\g\ltimes\F$-invariant foliate tubular neighbourhood of~$X$.  Given
two such tubular neighbourhoods $f_0$, $f_1\colon NX\to M$, we have a
corresponding pair of tubular neighbourhoods $f_{W,0}$, $f_{W,1}\colon
NX_W\to W$ of $X_W$.  There exists a $G\times K$-equivariant isotopy
$F_W\colon[0,1]\times NX_W\to W$ from $f_{W,0}$ to $f_{W,1}$, which
gives rise to a $\g$-equivariant foliate isotopy $F\colon[0,1]\times
NX\to M$ from $f_0$ to~$f_1$.
\end{proof}

\begin{remark}\label{remark;tubular}
For later use we mention some further properties of the normal bundle
$NX$.  The transverse metric $g$ on $(M,\F)$ restricts to a transverse
metric on $(X,\F_X)$ and to a fibre metric on $NX\cong N\F/N\F_X$.
The transverse Levi-Civita connection descends to a $\g$-invariant
metric connection on $NX$.  Similarly, the normal bundle $NX_W$ is
equipped with a $G\times K$-invariant metric and connection coming
from the Riemannian metric and Levi-Civita connection on $W$.  Given a
pair of embeddings $f\colon NX\to M$ and $f_W\colon NX_W\to W$ as in
the proof of the proposition, the isomorphism
$\pi^*NX\cong\varrho^*NX_W$ is an isomorphism of metric vector bundles
with connection.
\end{remark}

Equally useful is the existence of appropriate partitions of unity.
We call an open cover $\ca{U}$ of $M$ \emph{invariant} with respect to
a Lie algebroid $\liea$ over $M$ if every $U\in\ca{U}$ is
$\liea$-invariant.  We say that a partition of unity
$(\chi_U)_{U\in\ca{U}}$ subordinate to an $\liea$-invariant open cover
$\ca{U}$ is \emph{$\liea$-invariant} if each $\chi_U$ is constant
along every $\liea$-orbit.  The existence of $\F$-invariant partitions
of unity was established in~\cite[Lemma
  2.2]{belfi-park-richardson;hopf}.  Here is a $\g\ltimes\F$-invariant
extension of that result.

\begin{proposition}\label{proposition;unity}
Let $M$ be a connected manifold equipped with a complete Riemannian
foliation $(\F,g)$ and let $a\colon\g\to\X(M,\F,g)$ be an isometric
transverse action.  Let $\ca{U}$ be a $\g\ltimes\F$-invariant open
cover of $M$.  There exists a $\g\ltimes\F$-invariant partition of
unity subordinate to $\ca{U}$.
\end{proposition}

\begin{proof}
Let $U\in\ca{U}$.  The complement $M\backslash U$ is closed and
$\F$-invariant, and therefore it is $\barF$-invariant.  It follows
that $U$ itself is $\barF$-invariant, because the leaf closures
$\barF(x)$ decompose $M$ into disjoint subsets.  Therefore the open
subset $U_P=\pi^{-1}(U)$ of $P$ is $K$-invariant and
$\g\ltimes\barF_P$-invariant.  Hence $U_P=\varrho^{-1}(U_W)$ for a
unique open subset $U_W$ of $W$, namely $U_W=\varrho(U_P)$.  The set
$U_W$ is $G\times K$-invariant, where $G$ is as in
Proposition~\ref{proposition;induced-action}.  The correspondence
$U\leftrightarrow U_W$ gives us a cover $\ca{U}_W$ of $W$ by $G\times
K$-invariant open subsets.  Since $G$ acts isometrically, the closure
$\bar{G}$ of $G$ acts properly, so there exists a $G\times
K$-invariant partition of unity subordinate to $\ca{U}_W$, which
transports back to a $\g\ltimes\F$-invariant partition of unity
subordinate to $\ca{U}$.
\end{proof}

We obtain from this a Mayer-Vietoris theorem for equivariant basic de
Rham theory, the non-equivariant version of which is due
to~\cite[Theorem 2.3]{belfi-park-richardson;hopf}.  For a
$\g\ltimes\F$-invariant open cover $\ca{U}=\{U_i\}_{i\in I}$ of $M$ we
have the \emph{basic \v{C}ech-de Rham complex}
$\bigl(\check{C}(\ca{U}),\delta\bigr)$ associated with $\ca{U}$, which
is defined by
\[
\check{C}^p(\ca{U})=\prod_{i\in
  I^{p+1}}\Omega(U_i,\F)\quad\text{and}\quad (\delta\alpha)_i=
\sum_{q=0}^{p+1}
(-1)^q\alpha_{(i_0,i_1,\cdots,\hat{\imath}_q,\dots,i_{p+1})}|_{U_i}.
\]
Here $U_i$ denotes the intersection $U_{i_0}\cap U_{i_1}\cap\cdots\cap
U_{i_p}$ for a multi-index $i=(i_0,i_1,\dots,i_p)\in I^{p+1}$.  We
obtain an augmentation $\Omega(M,\F)\to\check{C}^0(\ca{U})$ by sending
the form $\alpha$ to the tuple $(\alpha|_{U_i})_{i\in I}$.  The basic
\v{C}ech-de Rham complex has a second grading coming from the
differential form degree, and a second differential, namely the
exterior derivative, which makes it a double complex
$\bigl(\check{C}(\ca{U}),\delta,d\bigr)$.  We denote the associated
total complex by $\check{C}_\tot(\ca{U})$.  Since $U_i$ is
$\g\ltimes\F$-invariant for each $i\in I^p$, the complex
$\Omega(U_i,\F)$ is a $\g$-differential graded algebra (as defined in
Section~\ref{section;basic-equivariant}), and therefore
$\check{C}^p(\ca{U})$ is a $\g$-differential graded algebra.  It
follows that the associated Weil complex
\[
\check{C}_\g^p(\ca{U})=(\W\g\otimes\check{C}^p(\ca{U}))_\bbas\g
\]
is a double complex $\bigl(\check{C}_\g(\ca{U}),\delta,d_\g\bigr)$,
which we call the \emph{$\g$-equivariant $\F$-basic \v{C}ech-de Rham
  complex}.  The following Mayer-Vietoris principle states that the
basic and the equivariant basic \v{C}ech-de Rham complexes are
homotopically trivial with respect to the \v{C}ech differential
$\delta$.
\glossary{CgU@$\check{C}_\g(\ca{U})$, $\g$-equivariant $\F$-basic
  \v{C}ech-de Rham complex}

\begin{proposition}[equivariant basic Mayer-Vietoris principle]
\label{proposition;mayer-vietoris}
Let $M$ be a connected manifold equipped with a complete Riemannian
foliation $(\F,g)$ and let $a\colon\g\to\X(M,\F,g)$ be an isometric
transverse action.  Let $\ca{U}=\{U_i\}_{i\in I}$ be a
$\g\ltimes\F$-invariant open cover of $M$.  The augmented basic
\v{C}ech-de Rham complex
\begin{equation}\label{equation;mayer-vietoris-basic}
\begin{tikzcd}
0\ar[r]&\Omega(M,\F)\ar[r]&\check{C}^0(\ca{U})\ar[r,"\delta"]&
\check{C}^1(\ca{U})\ar[r,"\delta"]&\cdots
\end{tikzcd}
\end{equation}
is a complex of $\g$-differential graded modules and is homotopically
equivalent to zero.  The augmented equivariant basic \v{C}ech-de Rham
complex
\begin{equation}\label{equation;mayer-vietoris-basic-equivariant}
\begin{tikzcd}
0\ar[r]&\Omega_\g(M,\F)\ar[r]&\check{C}_\g^0(\ca{U})\ar[r,"\delta"]&
\check{C}_\g^1(\ca{U})\ar[r,"\delta"]&\cdots
\end{tikzcd}
\end{equation}
is likewise homotopically equivalent to zero.  Therefore the
augmentation induces homotopy equivalences
\[
\begin{tikzcd}
\Omega(M,\F)\ar[r,"\simeq"]&\check{C}_\tot(\ca{U}),
\end{tikzcd}
\qquad
\begin{tikzcd}
\Omega_\g(M,\F)\ar[r,"\simeq"]&\check{C}_{\g,\tot}(\ca{U}),
\end{tikzcd}
\]
and isomorphisms
\[
\begin{tikzcd}
H(M,\F)\ar[r,"\cong"]&H\bigl(\check{C}_\tot(\ca{U})\bigr),
\end{tikzcd}
\qquad
\begin{tikzcd}
H_\g(M,\F)\ar[r,"\cong"]&H\bigl(\check{C}_{\g,\tot}(\ca{U})\bigr).
\end{tikzcd}
\]
\end{proposition}

\begin{proof}
The \v{C}ech differential $\delta$ is induced by inclusions of
$\g\ltimes\F$-invariant open subsets, and is therefore a morphism of
$\g$-differential graded modules.  The usual construction of a
null-homotopy $\kappa$ as in~\cite[Proposition
  8.5]{bott-tu;differential-forms} works: let $(\chi_i)_{i\in I}$ be a
$\g\ltimes\F$-invariant partition of unity subordinate to $\ca{U}$,
the existence of which is guaranteed by
Proposition~\ref{proposition;unity}, and let
$\alpha\in\check{C}^p(\ca{U})$.  For every $j\in I$ and $i\in I^{p-1}$
the form $\chi_j\alpha_{j,i_0,\dots,i_{p-1}}\in\Omega(U_j\cap U_i,\F)$
is supported on $\supp(\chi_j)\cap U_i$, so it extends by zero to a
unique form $\beta_{j,i}\in\Omega(U_i,\F)$.  Define
$\kappa(\alpha)\in\check{C}^{p-1}(\ca{U})$ by
$(\kappa(\alpha))_i=\sum_{j\in I}\beta_{j,i}$.  This sum is locally
finite as the family $(\supp(\chi_j)_{j\in I}$ is locally finite; so
$\kappa$ is well-defined and we have $[\delta,\kappa]=\id$.  Since the
$\chi_i$ are $\g\ltimes\F$-invariant, we have
$[\iota(\xi),\kappa]=[L(\xi),\kappa]=0$ for all $\xi\in\g$.  This
proves that~\eqref{equation;mayer-vietoris-basic} is null-homotopic as
a complex of $\g$-differential graded modules.  Being a homotopy of
$\g$-dgm, the map $\kappa$ extends to a map of Weil complexes
\[
\kappa_\g\colon\check{C}_\g^p(\ca{U})\longto\check{C}_\g^{p-1}(\ca{U}),
\]
which satisfies $[\delta,\kappa_\g]=\id$, thus showing
that~\eqref{equation;mayer-vietoris-basic-equivariant} is
null-homotopic as well.  The last assertion is a formal consequence of
the $\delta$-exactness; see~\cite[Proposition
  8.8]{bott-tu;differential-forms}.
\end{proof}

The next corollary is a variation on a result of
Molino~\cite[Proposition 3.7]{molino;riemannian-foliations} and is
mentioned without proof
in~\cite[\S\,3.5]{goertsches-toeben;equivariant-basic-riemannian}.

\begin{proposition}\label{proposition;leaf-orbifold}
Let $M$ be a connected manifold equipped with a complete Riemannian
foliation $(\F,g)$ and let $a\colon\g\to\X(M,\F,g)$ be an isometric
transverse action.  If the leaves of $\F$ are closed, the $K$-action
on $W$ is of constant infinitesimal orbit type, and therefore $M/\F$
is an orbifold.
\end{proposition}

\begin{proof}
Let $w\in W$.  Choose $p\in\varrho^{-1}(w)$ and let $x=\pi(p)$.  Let
$\Stab(K,w)$ be the stabilizer of $w$ under the $K$-action and let
$\stab(\liek,w)$ be its Lie algebra.  Then
\[
\stab(\liek,w)=\bigl\{\eta\in\liek\bigm|\eta_{P,p}\in
T_p\barF_P(p)\bigr\},
\]
and therefore the map $\eta\mapsto\eta_{P,p}$ is an isomorphism of
vector spaces
\[
\stab(\liek,w)\cong T_p(K{\cdot}p)\cap T_p\barF_P(p).
\]
Let $F_x=\pi^{-1}(\barF(x))$ and let $\pi_x\colon F_x\to\barF(x)$ be
the restriction of $\pi$ to $F_x$.  Then $F_x=K{\cdot}\barF_P(p)$ is
the $K$-orbit of the leaf closure $\barF_P(p)$, so
\[
T_pF_x=T_p(K{\cdot}p)+T_p\barF_P(p),\qquad
\ker(T_p\pi_x)=T_p(K{\cdot}p)\cap T_p\barF_P(p)\cong\stab(\liek,w).
\]
The hypothesis that the leaves of $\F$ are closed tells us that
$\dim(\barF(x))=\dim(\F(x))$ is constant, which yields that
$\dim(F_x)$ and $\dim(\stab(\liek,w))$ are constant.  This implies
that the Lie subalgebras $\stab(\liek,w)$ are all conjugate to one
another, i.e.\ all $w\in W$ are of the same infinitesimal orbit type.
Therefore, by Theorem~\ref{theorem;molino-1}\eqref{item;leaf-closure},
$M/\F=M/\barF=W/K$ is an orbifold.  (More precisely, in the language
of Theorem~\ref{theorem;molino-3}\eqref{item;molino-morita} the
generalized morphism $\Hol(M,\F)\to K\ltimes W$ defined by the
bibundle $P$ is a weak equivalence of groupoids, and therefore induces
an equivalence of \'etale stacks $[M/\F]\simeq[W/K]$.)
\end{proof}
\glossary{stab@$\stab(\liea,x)$, stabilizer Lie algebra of $x$ with
  respect to Lie algebra or Lie algebroid $\liea$}

The final result of this section, which extends~\cite[Proposition
  4.5]{goertsches-toeben;equivariant-basic-riemannian}, describes the
stabilizer Lie algebras of the Lie algebroid
$(\g\times\liec)\ltimes\F$ over $M$ in terms of the infinitesimal
stabilizers of the $G\times K$-action on the Molino manifold $W$.
Recall that $\theta_\LC$ denotes the transverse Levi-Civita connection
of $P$.  Recall also that for $x\in M$ the fibre $\liec_x$ of the
centralizer bundle $\liec$ consists of all germs at $x$ of transverse
Killing vector fields $\eta$ with the property that the natural lift
$\pi^\dag(\eta)$ commutes with all global transverse vector fields of
$P$.  Let us denote the value of $\eta$ at $x$ by $\eta_{M,x}\in
N_x\F$ and the value of $\pi^\dag(\eta)$ at $p\in P$ by $\eta_{P,p}\in
N_p\F_P$.  Since $\theta_\LC$ is $\F_P$-basic, for all $\xi\in\g$ and
$\eta\in\liec_x$ the expression
$\iota(\xi_{P,p}+\eta_{P,p})\theta_\LC$ is a well-defined element of
$T_p(K\cdot p)\cong\liek$.

\begin{proposition}[stabilizer correspondence]
\label{proposition;stabilizer-correspondence}
Let $M$ be a connected manifold equipped with a complete Riemannian
foliation $(\F,g)$ and let $a\colon\g\to\X(M,\F,g)$ be an isometric
transverse action.  Let $p\in P$, $x=\pi(p)\in M$ and $w=\varrho(p)\in
W$.  Define $\phi\colon\g\times\liec_x\to\g\times\liek$ by
\[
\phi(\xi,\eta)=\bigl(\xi,-\iota(\xi_{P,p}+\eta_{P,p})\theta_\LC\bigr).
\]
Let
\[
\liel_x=\stab((\g\times\liec)\ltimes\F,x)\subseteq\g\times\liec_x
\]
be the stabilizer of $x$ with respect to the
$(\g\times\liec)\ltimes\F$-action on $M$, and let
\[\liel_w=\stab(\g\times\liek,w)\subseteq\g\times\liek\]
be the stabilizer of $w$ with respect to the $\g\times\liek$-action on
$W$.  The restriction of $\phi$ to $\liel_x$ is a Lie algebra
isomorphism from $\liel_x$ onto~$\liel_w$.
\end{proposition}

\begin{proof}
Let $v$ be a transverse Killing vector field defined on an open
neighbourhood $U$ of $x$.  Its natural lift $\pi^\dag(v)$ is defined
on $\pi^{-1}(U)\subseteq P$ and decomposes into a vertical part
$\pi^\dag(v)_\ver$ and a horizontal part $\pi^\dag(v)_\hor$ with
respect to the transverse Levi-Civita connection.  Thus $v$ vanishes
at $x$ if and only if $\pi^\dag(v)$ is vertical at $p$ if and only if
$\pi^\dag(v)_\hor=0$.  Let us apply this observation to the transverse
Killing vector field $v=\xi_M+\eta_M$, where
$(\xi,\eta)\in\g\times\liec_x$.  Then $v$ is defined in a
neighbourhood of $x$ and vanishes at $x$ precisely when the pair
$(\xi,\eta)$ is in $\liel_x$.  Thus
\begin{equation}\label{equation;stabilizer}
(\xi,\eta)\in\liel_x\iff(\xi_{P,p}+\eta_{P,p})_\hor=0.
\end{equation}
Let $\phi(\xi,\eta)_W$ denote the vector field on $W$ induced by
$\phi(\xi,\eta)$.  Any $u\in T_pP$ defines an element
$\iota(u)\theta_\LC\in\liek$, which induces a vector field $u'$ on
$P$, whose value at $p$ is equal to $u'_p=u_\ver$.  Taking
$u=\xi_{P,p}+\eta_{P,p}$ we get $u'_p=(\xi_{P,p}+\eta_{P,p})_\ver$, so
that $\phi(\xi,\eta)_{W,w}=
\xi_{W,w}-T_p\varrho\bigl((\xi_{P,p}+\eta_{P,p})_\ver\bigr)$.  So if
$(\xi,\eta)$ is in $\liel_x$, it follows
from~\eqref{equation;stabilizer} that
\[
\phi(\xi,\eta)_{W,w}=\xi_{W,w}-T_p\varrho(\xi_{P,p}+\eta_{P,p})=
\xi_{W,w}-\xi_{W,w}=0,
\]
i.e.\ $\phi(\xi,\eta)\in\liel_w$.  This shows that
$\phi(\liel_x)\subseteq\liel_w$.  Conversely, let
$(\xi,\zeta)\in\liel_w$, i.e.\ $\xi_{W,w}+\zeta_{W,w}=0$.  Then the
transverse vector field $\xi_P+\zeta_P$ on $P$ is tangent to the fibre
$\varrho^{-1}(w)$.  It follows from Theorem~\ref{theorem;molino-2}
that $\xi_{P,p}+\zeta_{P,p}\in T_p\barF_P/T_p\F_P=\liec_{P,p}$, so
$\xi_{P,p}+\zeta_{P,p}=-\eta_{P,p}$ for some
$\eta\in\liec_{P,p}\cong\liec_x$.  Since $\zeta\in\liek$, the vector
field $\zeta_P$ is vertical, so
\[
(\xi_{P,p}+\eta_{P,p})_\ver=-(\zeta_{P,p})_\ver=
-\zeta_{P,p}=\xi_{P,p}+\eta_{P,p}.
\]
Therefore $(\xi,\eta)\in\liel_x$ by~\eqref{equation;stabilizer}, and
also $\phi(\xi,\eta)=(\xi,\iota(\zeta_{P,p})\theta_\LC)=(\xi,\zeta)$.
This shows that $\phi(\liel_x)\supseteq\liel_w$.  Next suppose
$(\xi,\eta)\in\liel_x$ satisfies $\phi(\xi,\eta)=0$.  Then $\xi=0$, so
$\iota(\eta_{P,p})\theta_\LC=0$, so $(\eta_{P,p})_\ver=0$, so
$\eta_{P,p}=0$ by~\eqref{equation;stabilizer}.  It follows that
$\eta=0$, because by
Theorem~\ref{theorem;molino-2}\eqref{item;algebroid} the centralizer
Lie algebroid $\liec\ltimes\F$ acts freely on $P$.  This shows that
the restriction of $\phi$ to $\liel_x$ is injective.  We finish by
showing that the restriction of $\phi$ to $\liel_x$ is a Lie algebra
homomorphism.  The map $\phi_0\colon\liel_x\to\liek$ defined by
$\phi_0(\xi,\eta)=\iota(\xi_{P,p}+\eta_{P,p})\theta_\LC$ is the
composition of three maps
\[
\begin{tikzcd}
\liel_x\ar[r,"\phi_1"]&\X(M,\F,g)_x\ar[r,"\pi^\dag"]&\X(K\cdot
p)^K\ar[r,"\phi_2"]&\liek,
\end{tikzcd}
\]
where $\X(M,\F,g)_x$ denotes the Lie algebra of Killing vector fields
on $M$ that vanish at $x$, and
$\phi_1(\xi,\eta)=\xi_{M,x}+\eta_{M,x}$, and $\phi_2$ is the bijective
map that sends a $K$-invariant vector field $u$ tangent to $K\cdot p$
to $\iota(u)\theta_{\LC,p}$.  Note that $\pi^\dag$ takes values in the
$K$-invariant vector fields (see~\eqref{equation;killing}) and maps
vector fields vanishing at $x$ to vector fields tangent to $K\cdot p$
because of~\eqref{equation;stabilizer}.  The maps $\phi_1$ and
$\pi^\dag$ are homomorphisms and $\phi_2$ is an anti-homomorphism, so
$\phi(\xi,\eta)=(\xi,-\phi_0(\xi,\eta))$ is a homomorphism.
\end{proof}

\section{The orbit type stratification}
\label{section;orbit-type}

\numberwithin{equation}{section}

In this section we show that an isometric transverse Lie algebra
action on a Riemannian foliated manifold can be locally linearized and
gives rise to a tidy stratification of the manifold by orbit type.
These facts were partly known to
Kobayashi~\cite{kobayashi;fixed-isometry}, and have been long
established for smooth proper Lie group actions, as set forth
e.g.\ in~\cite[\S\,IX.9]{bourbaki;groupes-algebres}.  A little caution
is needed: orbits of Lie algebra actions may not have tubular
neighbourhoods (e.g.\ Kronecker's dense line in the torus), so the
``slice theorem'' is valid only in a weak form,
Proposition~\ref{proposition;tube} below.

\subsection*{Notation and conventions}

See \S\S\,\ref{section;foliation}--\ref{section;transverse} and the
notation index in the back for our general conventions regarding
foliations and transverse Lie algebra actions.  We denote the Lie
algebra of Killing vector fields of a Riemannian manifold $(S,g_S)$ by
$\X(S,g_S)$.  Throughout \S\,\ref{section;orbit-type} $\g$ denotes a
finite-dimensional Lie algebra and $M$ denotes a manifold equipped
with a Riemannian foliation $(\F,g)$ and an isometric transverse
action $a\colon\g\to\X(M,\F,g)$.  We will assume that $(M,\F,g)$ is
metrically complete in the sense of \S\,\ref{section;molino}, so that
the Molino structure theory applies, although many of the results of
this section hold without this assumption.  For $\xi\in\g$ we denote
the transverse vector field $a(\xi)$ by $\xi_M$.  By a submanifold of
$M$ we will mean an embedded submanifold.  We will allow our
submanifolds to be ``impure'', meaning that they may have connected
components of different dimensions.

\subsection*{The stratification}

We start with the simple observation that every foliation chart of
$(M,\F)$ is $\g$-equivariant in a suitable sense.  This will allow us
for most purposes to disregard the foliation and restrict our
attention to the transverse directions.  Recall that the transverse
metric $g$ restricts to an ordinary Riemannian metric $g_S$ on any
transversal $S$ to the foliation, and that the transverse action $a$
restricts to an ordinary Lie algebra action
$a_S\colon\g\to\X(S,g_S)$ on $S$ by Killing vector fields.

\begin{lemma}\label{lemma;equivariant-foliation-chart}
Let $x\in M$.  Put $T=T_xM$ and $F=T_x\F$.  Choose a transversal $S$
to the foliation $\F$ at $x$.  Let $g_S$ be the Riemannian metric on
$S$ induced by $g$ and let $a_S\colon\g\to\X(S,g_S)$ be the
isometric $\g$-action on $S$ induced by $a$.  Equip the product
$S\times F$ with the foliation given by the fibres of the projection
$S\times F\to S$ and with the transverse $\g$-action defined by the
action $a_S$ on the first factor.  Let $B$ be an open neighbourhood of
$x$ in $S$ and let $\gamma\colon B\times F\to M$ be a foliate open
embedding satisfying $\gamma(y,0)=y$ for all $y\in B$, and
$T_x\gamma=\id_T$.  Then $\gamma$ is $\g$-equivariant.
\end{lemma}

\begin{proof}
This follows from the fact that the transverse $\g$-action on the
foliation chart $\gamma(B\times F)$ is determined by the $\g$-action
on the transversal $B$, as discussed in \S\,\ref{section;transverse}.
\end{proof}

The following result says that the $\g$-action can be linearized at a
fixed leaf, which is the analogue of the Bochner linearization
theorem.  A subspace $W$ of a vector space $V$ defines a foliation of
$V$, whose leaves are the affine subspaces parallel to $W$ and whose
leaf space is the quotient $V/W$.  We call this the \emph{linear
  foliation} of $V$ defined by $W$.

\begin{proposition}\label{proposition;fixed-leaf}
Let $x\in M$ and $\h=\stab(x,\g\ltimes\F)$.  Let $\F_T$ be the linear
foliation of the tangent space $T=T_xM$ defined by the linear subspace
$F=T_x\F$.
\begin{enumerate}
\item\label{item;inner}
The inner product $g_x$ on $T/F$ defines a transverse Riemannian
metric on $(T,\F_T)$.  The Lie algebra $\h$ acts transversely on
$(T,\F_T)$ by linear infinitesimal isometries.
\item\label{item;bochner}
There is an $\h$-equivariant foliate open embedding $\psi\colon T\to
M$ with the properties $\psi(0)=x$ and $T_0\psi=\id_T$.
\end{enumerate}
\end{proposition}

\begin{proof}
\eqref{item;inner}~The quotient vector space $T/F$ is the fibre at $x$
of the normal bundle $N\F$ of the foliation, so it carries the inner
product $g_x$.  The normal bundle of the linear foliation $\F_T$ is
the trivial bundle over $T$ with fibre $T/F$, so the inner product
$g_x$ defines a translation-invariant transverse Riemannian metric on
the foliated manifold $(T,\F_T)$.  The Lie algebra of isometric linear
transverse vector fields on $(T,\F_T)$ is equal to $\lie{o}(T/F)$, the
Lie algebra of skew-symmetric linear endomorphisms of the
inner-product space $T/F$.  For $\xi\in\h$ the transverse vector field
$\xi_M$ vanishes at $x$, so $\xi_T=\ad\xi_{M,x}$ is a well-defined
linear endomorphism of the fibre $N_x\F=T/F$, which is skew-symmetric
because $\xi_M$ preserves $g$.  The map $a_x\colon\h\to\lie{o}(T/F)$
that sends $\xi$ to $\xi_T$ defines a linear isometric transverse
$\h$-action $(T,\F_T)$.

\eqref{item;bochner}~Choose a transversal $S$ to the foliation $\F$ at
$x$.  Let $g_S$ be the metric on $S$ induced by the transverse metric
$g$ and $a_S\colon\g\to\X(S,g_S)$ the $\g$-action on $M$ induced
by the transverse $\g$-action $a$.  Since $x\in S$ is $\h$-fixed, for
$\xi\in\h$ each trajectory of the vector field $\xi_S=a_S(\xi)$ stays
at a fixed distance from $x$, which shows that the $\xi_S$ are
complete on a sufficiently small ball $B$ about $x$.  Let $\liek$ be
the image of $\h$ in $\X(S,g_S)$.  By the Lie-Palais theorem the
$\liek$-action exponentiates to an action of $K$, a connected immersed
Lie subgroup of the isometry group of $(S,g_S)$ with Lie algebra
$\liek$.  Then the closure $\bar{K}$ of $K$ in the isometry group
fixes $x$ and is therefore compact, and we have Lie algebra
homomorphisms $\h\to\liek\to\bar{\liek}$, where
$\bar{\liek}=\Lie(\bar{K})$.  Let $E=T_xS$.  Then $E$ is a subspace of
$T$ complementary to $F$, so it is an inner product space isomorphic
to $T/F$, and the linearization of the $\bar{K}$-action at $x$ is a
homomorphism $\bar{\liek}\to\lie{o}(E)$.  By the usual Bochner
linearization theorem (see e.g.~\cite[\S\,IX.9, Proposition
  5]{bourbaki;groupes-algebres}) there exists (after replacing $B$ by
a smaller ball if necessary) a $\bar{K}$-equivariant diffeomorphism
$\phi\colon B\to E$ mapping $x$ to $0$ and with derivative
$T_0\phi=\id_E$.  Since the map $\phi$ is $\bar{K}$-equivariant, it is
also $\h$-equivariant.  After again shrinking $B$ if necessary, we
have a foliation chart $\gamma\colon B\times F\to M$ as in
Lemma~\ref{lemma;equivariant-foliation-chart}, and we define
$\psi\colon T=E\oplus F\to M$ by $\psi(e,f)=\gamma(\phi^{-1}(e),f)$.
The map $\psi$ satisfies our requirements.
\end{proof}

The next result is a $\g$-equivariant normal form for a neighbourhood
of an arbitrary point $x$ of $M$.  Unlike the slice theorem of the
theory of proper transformation groups (see e.g.~\cite[\S\,IX.9,
  Proposition 6]{bourbaki;groupes-algebres}), it does not give a
normal form for a tubular neighbourhood of the entire
$\g\ltimes\F$-orbit of $x$, but only for a portion of the orbit.  Let
$\h=\stab(x,\g\ltimes\F)$ be the stabilizer of $x$.  Choose a
transversal $S$ at $x$ to the foliation $\F$, let
$a_S\colon\g\to\X(S,g_S)$ be the induced $\g$-action on $S$, and
let $\exp(t\xi_S)$ denote the flow of the vector field $\xi_S$ on $S$
induced by $\xi\in\g$.  Let $E$ be the inner-product space $T_xS$.
Let $V=\{\,\xi_{S,x}\mid\xi\in\g\,\}^\perp$ be the orthogonal
complement in $E$ of the tangent space to the $\g$-orbit of $x$.  The
map $\psi$ of Proposition~\ref {proposition;fixed-leaf} maps $V$ to a
submanifold of $S$ which is complementary to the $\g$-orbit of $x$.
The action of $\h$ on $T=T_xM$ preserves the subspaces $E\subseteq T$
and $V\subseteq E$.  Let $G$ be a Lie group with Lie algebra $\g$.
Define an action $\g\times\h\to\X(G\times V)$ by
\begin{equation}\label{equation;model-action}
(\xi,\eta)\longmapsto(\xi_R-\eta_L,\eta_V),
\end{equation}
where $\xi_L$ (resp.\ $\xi_R$) denotes the left-invariant
(resp.\ right-invariant) vector field on $G$ induced by $\xi\in\g$,
and $\eta_V$ denotes the linear vector field on $V$ induced by
$\eta\in\h$.  The action of the second factor $\h$ is free, so the
$\h$-orbits form a foliation $\ca{H}$ of $G\times V$.  The $\g$-action
preserves the foliation $\ca{H}$ and therefore descends to a
transverse action $\g\to\X((G\times V)/\ca{H})$.  The map
\[
\beta\colon B_G\times B_V\longto
S,\qquad(\exp_G(\xi),v)\longmapsto\exp(\xi_S)(\psi(v))
\]
is well-defined for sufficiently small open neighbourhoods $B_G$ of
$1\in G$ and $B_V$ of $0\in V$.  The map $\beta$ is $\g$-equivariant
and constant along the leaves of $\ca{H}$.  Its tangent map
$T_{(1,0)}\beta\colon\g\times V\to E$ is given by
\begin{equation}\label{equation;tangent}
T_{(1,0)}\beta(\xi,v)=\xi_{S,x}+v
\end{equation}
for $\xi\in\g$ and $v\in V$, so $\beta$ is a submersion at $(1,0)$.
Hence, after replacing $B_G$ and $B_V$ by smaller open subsets if
necessary, the image $B=\beta(B_G\times B_V)$ is an open neighbourhood
of $x$ in $S$.  Let us take $B_G$ and $B_V$ so small that we have a
foliation chart $\gamma\colon B\times F\to M$ as in
Lemma~\ref{lemma;equivariant-foliation-chart}, and define
\begin{equation}\label{equation;tube}
\alpha\colon B_G\times B_V\times F\longto
M,\qquad(\exp_G(\xi),v,f)\longmapsto
\gamma\bigl(\beta\bigl(\exp(\xi_S)(\psi(v))\bigr),f\bigr).
\end{equation}

\begin{proposition}\label{proposition;tube}
Let $x\in M$.  Put $\h=\stab(x,\g\ltimes\F)$ and $F=T_x\F$.  Let $G$
be a Lie group with Lie algebra $\g$ and define a $\g\times\h$-action
on $G\times V$ by~\eqref{equation;model-action}.  Let $\ca{H}$ be the
foliation of $G\times V$ defined by the $\h$-action, and let
$X\subseteq B_G\times B_V$ be a transversal at $(1,0)$ to $\ca{H}$.
Let $\alpha$ be the map defined in~\eqref{equation;tube}.  There
exists an open neighbourhood $B_X$ of $(1,0)$ in $X$ with the property
that $\alpha\colon B_X\times F\to M$ is a $\g$-equivariant foliate
open embedding which maps $(1,0,0)$ to $x$.
\end{proposition}

\begin{proof}
The property $\alpha(1,0,0)=x$ is evident from~\eqref{equation;tube}.
Since $X$ is a transversal to the foliation $\ca{H}$, the transverse
$\g$-action on $G\times V$ restricts to a $\g$-action on $X$, and
$\beta\colon X\to S$ is $\g$-equivariant.  It follows
from~\eqref{equation;tangent} that the kernel of
$T_{(1,0)}\beta\colon\g\times V\to E$ is equal to
$\h\times\{0\}=T_{(1,0)}\ca{H}$.  Therefore $\beta\colon X\to S$ is
\'etale at $(1,0)$, so $\beta$ restricts to a $\g$-equivariant open
embedding $\beta\colon B_X\to S$ on a sufficiently small open
neighbourhood $B_X\subseteq X$ of $x$.
Lemma~\ref{lemma;equivariant-foliation-chart} now shows that
$\alpha\colon B_X\times F\to M$ is a $\g$-equivariant open embedding.
\end{proof}

Let $\h$ be a Lie subalgebra of $\g$.  We say that $x\in M$ is of
\emph{symmetry type} $\h$ if the stabilizer $\stab(x,\g\ltimes\F)$ is
equal to $\h$, and we define
\[
M_\h=\{\,x\in M\mid\text{$x$ is of symmetry type $\h$}\,\}.
\]
We denote by $(\h)$ the collection of all Lie subalgebras of $\g$ that
are conjugate to $\h$ under the adjoint group $\Ad(\g)$.  We say that
$x\in M$ is of \emph{orbit type} $(\h)$ if the stabilizer subalgebra
$\stab(x,\g\ltimes\F)$ is an element of $(\h)$, and we define the
\emph{stratum of orbit type $(\h)$} to be
\[
M_{(\h)}=\{\,x\in M\mid\text{$x$ is of orbit type $(\h)$}\,\}.
\]
There are obvious inclusions $M_\h\subseteq M^\h$ and $M_\h\subseteq
M_{(\h)}$.  We will now deduce from the local model theorem,
Proposition~\ref{proposition;tube}, that the sets $M^\h$, $M_\h$, and
$M_{(\h)}$ are submanifolds of $M$.  Since they are $\F$-invariant, it
then follows from Remark~\ref{remark;invariant} that they are
automatically invariant under the centralizer Lie algebroid
$\liec\ltimes\F$, i.e.\ they are unions of leaf closures.

\begin{theorem}\label{theorem;symmetry}
Let $\h$ be a Lie subalgebra of $\g$ with normalizer $\n=\n_\g(\h)$.
\begin{enumerate}
\item\label{item;fixed}
The fixed-leaf set $M^\h$ is an $\n\ltimes\F$-invariant closed
submanifold of $M$.
\item\label{item;symmetry}
The symmetry type set $M_\h$ is an $\n\ltimes\F$-invariant open
submanifold of $M^\h$.
\item\label{item;orbit}
The orbit type stratum $M_{(\h)}$ is a $\g\ltimes\F$-invariant
submanifold of $M$.
\end{enumerate}
\end{theorem}

\begin{proof}
\eqref{item;fixed}~The $\n\ltimes\F$-invariance follows from
Lemma~\ref{lemma;stabilizer} and
Proposition~\ref{proposition;closure-stabilizer}%
\eqref{item;closure-stabilizer}.  Every subalgebra $\h$ acts
isometrically transversely on $M$, so it is enough to prove the rest
of the assertion for $\h=\g$.  That $M^\g$ is closed is evident from
the smoothness of the transverse vector fields $\xi_M$.  Let $x\in
M^\g$.  We compute $M^\g$ in the chart at $x$ given by
Proposition~\ref{proposition;fixed-leaf}.  The tangent space $T=T_xM$
is a direct sum $T=E\oplus F$, where $E$ is an inner-product space
with a linear isometric $\g$-action and the foliation of $T$ is the
linear foliation defined by $F$.  Thus $T^\g=E^\g\oplus F$, where
$E^\g$ is the $\g$-fixed subspace of $E$.  It follows that $M^\g$ is a
submanifold.

\eqref{item;symmetry}~The $\n\ltimes\F$-invariance follows from
Lemma~\ref{lemma;stabilizer} and
Proposition~\ref{proposition;closure-stabilizer}%
\eqref{item;closure-stabilizer}.  Let $x\in M_\h$.  To show that
$M_\h$ is open in $M^\h$ it suffices to show that $M_\h\cap U=M^\h\cap
U$ for some open neighbourhood $U$ of $x$.  We take $U$ to be of the
form $U=\alpha(B_X\times F)$ as in Proposition~\ref{proposition;tube},
which allows us to replace $M$ with $X\times F$.  We have $(X\times
F)_\h=X_\h\times F$ and $(X\times F)^\h=X^\h\times F$, so it is enough
to show that $X_\h=X^\h$.  The manifold $X$ is a transversal to the
foliation $\ca{H}$ of $G\times V$, so the $\g$-stabilizer of a point
$(k,v)$ in the $\g$-manifold $X$ is the same thing as the stabilizer
of $(k,v)$ viewed as a point in $G\times V$ with respect to the
transverse action Lie algebroid $\g\ltimes\ca{H}$ over $G\times V$.
But this Lie algebroid can be integrated as follows.  Let $H$ be the
immersed connected subgroup of $G$ corresponding to the Lie subalgebra
$\h\subseteq\g$, and let $\zeta\colon\tilde{H}\to H$ be a covering
group with the property that the action $\h\to\lie{o}(V)$ lifts to an
action $\tilde{H}\to\SO(V)$.  Then $G\times\tilde{H}$ acts on $G\times
V$ by $(g,\tilde{h})\cdot(k,v)=(gk\zeta(h)^{-1},hv)$.  The
corresponding action groupoid $(G\times\tilde{H})\ltimes(G\times V)$
has source and target maps defined by $s(g,\tilde{h},k,v)=(k,v)$ and
$t(g,\tilde{h},k,v)=(gk\zeta(\tilde{h})^{-1},\tilde{h}v)$, and its Lie
algebroid is isomorphic to $\g\ltimes\ca{H}$.  The stabilizer of
$(k,v)$ relative to this groupoid is
\[
\Stab(k,v)=k\Stab(v,\tilde{H})k^{-1},
\]
where $\Stab(v,\tilde{H})$ denotes the stabilizer of $v\in V$ with
respect to the linear $\tilde{H}$-action, and where $G$ acts on
$\tilde{H}$ by conjugation.  Infinitesimally we have
\begin{equation}\label{equation;model-stabilizer}
\stab((k,v),\g\ltimes\ca{H})=\Ad_k(\stab(v,\h))
\end{equation}
for $k\in G$ and $v\in V$.  In particular $(k,v)$ is of symmetry type
$\h$ if and only if $k$ is in the normalizer $N_G(\h)$ and $v\in
V^\h$.  Thus $(G\times V)_\h=(G\times V)^\h=N_G(\h)\times V^\h$, which
proves that $X_\h=X^\h$.

\eqref{item;orbit}~The $\g\ltimes\F$-invariance follows from
Lemma~\ref{lemma;stabilizer} and
Proposition~\ref{proposition;closure-stabilizer}%
\eqref{item;closure-stabilizer}.  In the local model $X\times F$ we
have $(X\times F)_{(\h)}=X_{(\h)}\times F$.
From~\eqref{equation;model-stabilizer} we get $X_{(\h)}=X\cap(G\times
V)_{(\h)}=X\cap(G\times V^\h)$.  By assumption $X$ is transverse to
the leaf $H\times\{0\}$ of $\ca{H}$.  It follows that $X$ is
transverse to $G\times V^\h$, so $X_{(\h)}$ is a submanifold of $X$.
This shows that $M_{(\h)}$ is a submanifold of $M$.
\end{proof}

On each orbit type stratum $M_{(\h)}$ the anchor map
$t\colon\g\ltimes\F\to TM$ has constant rank equal to
$\dim(\g)-\dim(\h)+\dim(\F)$ and therefore defines a regular
foliation.  We finish by showing that the orbit-type stratification is
locally finite.  We partially order the set of conjugacy classes of
Lie subalgebras of $\g$ by declaring
\begin{equation}\label{equation;order}
(\h_1)\le(\h_2)
\end{equation}
if $\h_2$ is conjugate to a subalgebra of $\h_1$.

\begin{theorem}\label{theorem;semicontinuous}
Every $x\in M$ has a $\g\ltimes\F$-invariant open neighbourhood $U$
with the property that for all $y\in U$ the orbit type of $y$ is
greater than or equal to the orbit type of $x$, and that $U\cap
M_{(\h)}$ is empty for all but finitely many conjugacy classes $(\h)$.
In particular, if $M/\barF$ is compact, the collection of $(\h)$ for
which $M_{(\h)}$ is nonempty is finite.
\end{theorem}

\begin{proof}
Let $x\in M$ and $\h=\stab(x,\g\ltimes\F)$.  The projection $M\to
M/\barF$ is open, so it is enough to show that $x$ has an open
neighbourhood $U$ such that (a)~$\bigl(\stab(y,\g\ltimes\F)\bigr)\ge
\bigl(\stab(x,\g\ltimes\F)\bigr)$ for all $y\in U$, and
(b)~$U$~intersects only finitely many orbit type strata $M_{(\liek)}$.
It follows from~\eqref{equation;model-stabilizer} that the stabilizer
of every point in the model manifold $X$ is conjugate to a subalgebra
of $\h$, so assertion~(a) follows from
Proposition~\ref{proposition;tube}.  Assertion~(b) is proved by
induction on the codimension $q$ of the foliation $\F$.  The case
$q=0$ is trivial.  Suppose the assertion is proved for all Lie
algebras $\liea$ and all Riemannian foliated manifolds $(Y,\F_Y)$
equipped with isometric transverse $\liea$-actions, where $\F_Y$ is of
codimension less than $q$.  Let us deduce that the assertion is true
for $M$.  Again computing in the model manifold $X$ at $x$, for each
subalgebra $\liek$ of $\g$ we have $X_{(\liek)}=\emptyset$ if
$(\liek)\not\ge(\h)$, and
\[X_{(\liek)}=G\times V^\h\times W_{(\liek)}\]
if $(\liek)\ge(\h)$, where $W$ is the orthogonal complement of $V^\h$
in $V$.  The orbit type stratum $W_{(\liek)}$ of the orthogonal
$\h$-module $W$ is preserved by multiplication by nonzero scalars, and
$W_{(\h)}=W^\h=\{0\}$.  Hence $W_{(\liek)}$ for $\liek\ne\h$
intersects the unit sphere $\group{S}(W)$, which is an $\h$-invariant
submanifold of $W$.  In other words, the set $\{\,(\liek)>(\h)\mid
X_{(\liek)}\ne\emptyset\,\}$ is in bijective correspondence with the
set $\{\,(\liek)>(\h)\mid\group{S}(W)_{(\liek)}\ne\emptyset\,\}$.
Since $\group{S}(W)$ is compact and has dimension $<q$, the induction
hypothesis tells us that $\group{S}(W)$, and hence $X$, contains only
finitely many orbit types, which completes the induction step.
\end{proof}

We mention without proof the following ``regular orbit type'' theorem.
The proof is similar to that of Theorem~\ref{theorem;semicontinuous}.

\begin{theorem}
If $M$ is connected, the set of orbit types of elements of $M$ has a
greatest element $(\h_0)$ with respect to the partial
order~\eqref{equation;order}.  The corresponding stratum $M_{(\h_0)}$
is connected, open, and dense.
\end{theorem}

\section{Borel-Atiyah-Segal localization}\label{section;localization}

In this section we derive from the stratification theorems of
Section~\ref{section;orbit-type} a ``contravariant'' localization
theorem in equivariant basic cohomology, Theorem~\ref{theorem;borel},
akin to the Borel-Atiyah-Segal theorem of ordinary equivariant
cohomology theory.  This answers a question posed in~\cite[Remark
  3.22]{goertsches-nozawa-toeben;chern-simons-foliations}.  The
important case where the Lie algebra is Molino's structural Lie
algebra of a Killing foliation was handled earlier
in~\cite[\S\,5]{goertsches-toeben;equivariant-basic-riemannian}.

\subsection*{Notation and conventions}

Throughout this section $\g$ denotes a finite-dimensional abelian Lie
algebra and $M$ denotes a connected manifold equipped with a complete
Riemannian foliation $(\F,g)$ and an isometric transverse action
$a\colon\g\to\X(M,\F,g)$.  For $\xi\in\g$ we denote the transverse
vector field $a(\xi)$ by $\xi_M$.

\subsection*{Contravariant localization}

The following theorem is a foliated version of results due to Borel,
Atiyah and Segal, which can be found
in~~\cite[\S\,3]{atiyah-bott;moment-map-equivariant-cohomology},
\cite[\S\,III.2]{hsiang;cohomology-theory;;1975},
or~\cite[\S\,III.3]{tom-dieck;transformation-groups}.  Let $i_\g\colon
M^\g\to M$ be the inclusion of the fixed-leaf set, which by
Theorem~\ref{theorem;symmetry}\eqref{item;fixed} is a
$\g\ltimes\F$-invariant closed submanifold of $M$.  Recall that, $\g$
being abelian, the equivariant basic cohomology groups $H_\g(M,\F)$
and $H_\g(M^\g,\F)$ (defined in
Section~\ref{section;basic-equivariant}) are modules over the
symmetric algebra $S\g^*$.

\begin{theorem}\label{theorem;borel}
Let $\g$ be abelian and $M/\barF$ compact.  The kernel and the
cokernel of the restriction homomorphism $i_\g^*\colon H_\g(M,\F)\to
H_\g(M^\g,\F)$ have support in the cone $\Gamma_M=\bigcup_{x\in
  M\backslash M^\g}\stab(\g\ltimes\F,x)$.
\end{theorem}

\begin{proof}
Since $M/\barF$ is compact, the collection of subalgebras $\h$ for
which $M_\h$ is nonempty is finite by
Theorem~\ref{theorem;semicontinuous}.  Hence the cone $\Gamma_M$ is
the union of a nonempty finite collection of proper linear subspaces
$\h_j$.  For each $j$ choose a nonzero $f_j\in\g^*$ which vanishes on
$\h_j$ and let $f=\prod_jf_j\in S\g^*$.  It is enough to show that
$i_\g^*$ becomes an isomorphism after inverting $f$, i.e.\ after
localizing $S\g^*$ at the multiplicative set $S=\{\,f^n\mid
n\in\N\,\}$.  Let $U$ be a $\g\ltimes\F$-invariant tubular
neighbourhood of $M^\g$ (the existence of which is guaranteed by
Proposition~\ref{proposition;tubular}) and let $V=M\backslash M^\g$.
Choose $\xi\in\g$ satisfying $f(\xi)\ne0$.  Then
$\xi\in\g\backslash\Gamma_M$, so $\xi_{M,x}\ne0$ for all $x\in V$,
which is to say that the transverse action of the Lie subalgebra
$\liek=\R\xi$ of $\g$ on $V$ is free.  Let $\h$ be a hyperplane
contained in the zero locus of $f$; then $\g=\liek\oplus\h$.  Let $\A$
be the commutative $\g$-differential graded algebra $\Omega(V,\F|_V)$.
The transverse $\liek$-action on $V$ is free, so
by~\cite[Lemma~3.18]{goertsches-toeben;equivariant-basic-riemannian}
$\A$ possesses an $\h$-invariant $\liek$-connection.  Applying
Theorem~\ref{theorem;equivariant-ccw} to $\A$ we obtain a homotopy
equivalence $\A_\g\simeq(\A_\bbas{\liek})_\h$,
i.e.\ $\A_\g\simeq\B_\bbas{\h}$, where
\[\B=\W\h\otimes\A_\bbas{\liek}.\]
Since $f$ vanishes on $\h$, multiplication by $f$ kills the subalgebra
$\W\h$ of $\B$.  The subalgebra $\A_\bbas{\liek}$ of $\B$ is a
subcomplex of the de Rham complex of the finite-dimensional manifold
$V$, and $f$ has positive degree, so a sufficiently high power of $f$
annihilates $\A_\bbas{\liek}$.  Hence $S^{-1}\B=0$ and
$S^{-1}\B_\bbas{\h}=0$.  It follows that the localization of
$H_\g(V,\F|_V)=H_\g(\A)\cong H(\B_\bbas{\h})$ at $S$ is equal to $0$.
By the same argument, the localization of $H_\g(U\cap V,\F|_{U\cap
  V})$ at $S$ is equal to $0$.  The equivariant basic Mayer-Vietoris
sequence
\[
\begin{tikzcd}[column sep=small]
\cdots\ar[r]&H_\g(M,\F)\ar[r]&H_\g(U,\F|_U)\oplus
H^*_\g(V,\F|_V)\ar[r]&H_\g(U\cap V,\F|_{U\cap V})\ar[r]&\cdots
\end{tikzcd}
\]
is exact (Proposition~\ref{proposition;mayer-vietoris}), so after
localizing at $S$ we get an isomorphism
\[
\begin{tikzcd}[column sep=small]
i_\g^*\colon S^{-1}H_\g(M,\F)\ar[r,"\cong"]&
S^{-1}H_\g(U,\F|_U)\ar[r,"\cong"]&S^{-1}H_\g(M^\g,\F|_{M^\g}),
\end{tikzcd}
\]
where the second isomorphism follows from
Lemma~\ref{lemma;cochain-homotopy}.
\end{proof}

\begin{corollary}\label{corollary;rank}
Let $\g$ be abelian and $M/\barF$ compact.  Then
\begin{align*}
\rank(H_\g^\even(M,\F))&=\dim(H^\even(M^\g,\F)),\\
\rank(H_\g^\odd(M,\F))&=\dim(H^\odd(M^\g,\F)).
\end{align*}
\end{corollary}

\begin{proof}
Let $S$ be the multiplicative set $S=S\g^*\backslash\{0\}$.  Then
$\FF=S^{-1}(S\g^*)$ is the fraction field of $S\g^*$.  By definition
the rank of an $S\g^*$-module $\M$ is equal to the dimension of the
$\FF$-vector space $S^{-1}\M$.  It follows
fromTheorem~\ref{theorem;borel} that $S^{-1}H_\g(M,\F)\cong
S^{-1}H_\g(M^\g,\F)$, so the $S\g^*$-modules $H_\g(M,\F)$ and
$H_\g(M^\g,\F)$ have the same rank.  Because the $\g$-action on $M^\g$
is trivial, we have $H_\g(M^\g,\F)\cong S\g^*\otimes H(M^\g,\F)$, so
the module $H_\g(M^\g,\F)$ is free of rank equal to the dimension of
the $\R$-vector space $H(M^\g,\F)$.  Therefore
$\rank(H_\g(M,\F))=\dim(H(M^\g,\F))$.  The same argument applies to
the even and odd parts, because localization of a $\Z$-graded
$S\g^*$-module at an ideal preserves the $\Z/2\Z$-grading.
\end{proof}

\section{Atiyah-Bott-Berline-Vergne localization}\label{section;abbv}

\numberwithin{equation}{subsection}

In this section we derive from the Borel-Atiyah-Segal localization
theorem, Theorem~\ref{theorem;borel}, a ``covariant'' localization
formula in equivariant basic cohomology,
Theorem~\ref{theorem;covariant}.  This uses the equivariant basic Thom
isomorphism theorem of our paper~\cite{lin-sjamaar;thom}.  Pushing
forward to a point then yields an integration formula,
Theorem~\ref{theorem;abbv}, akin to the formulas of Atiyah and
Bott~\cite{atiyah-bott;moment-map-equivariant-cohomology} and Berline
and Vergne~\cite{berline-vergne;zeros}.  An immediate consequence is a
Duistermaat-Heckman formula for transversely symplectic foliations,
Theorem~\ref{theorem;duistermaat-heckman}.

\subsection{Notation and conventions}

In \S\,\ref{section;abbv} we denote by $\g$ a finite-dimensional Lie
algebra and by $M$ a manifold equipped with a complete Riemannian
foliation $(\F,g)$ of codimension $q$ and an isometric transverse
action $a\colon\g\to\X(M,\F,g)$.  For $\xi\in\g$ we denote the
transverse vector field $a(\xi)$ by $\xi_M$.

If $(C,d)$ is a cochain complex we write $(C[k],d[k])$ for the
$k$-shifted complex, which is defined by $C[k]^i=C^{i+k}$ and
$d[k]=(-1)^kd$.  See also the notation index at the end.
\glossary{.@$[\cdot]$, translation functor}
%

\subsection{Pushforward homomorphisms and transverse integration}
\label{section;transverse-integration}

We review the definition of the pushforward homomorphism (also called
wrong-way or Thom-Gysin homomorphism) in equivariant basic de Rham
theory given
in~\cite[Appendix~A]{goertsches-nozawa-toeben;chern-simons-foliations},
and we state a few of its properties.  In particular we show how the
fundamental class in equivariant basic cohomology of a submanifold is
represented by the equivariant basic Thom form of its normal bundle.
These facts make use of the notion of a transverse integral of a basic
differential form, which is an integral ``across'' the leaves of a
Riemannian foliation.  Unlike the integral of a differential form over
a manifold, the transverse integral is not canonically defined, but
depends on choices.  However, it does satisfy Stokes' formula and
Poincar\'e duality.

Let $X$ be a co-oriented $\g\ltimes\F$-invariant closed submanifold of
$M$ of codimension~$r$.  Let $i=i_X\colon X\to M$ be the inclusion
map, let $\pi=\pi_X\colon NX\to X$ be the normal bundle, and let
$\zeta=\zeta_X\colon X\to NX$ be the zero section.  The normal bundle
is a foliated bundle equipped with a foliation $\F_{NX}$, and we
denote by $\Omega_{\g,\cv}(NX,\F_{NX})$ the complex of equivariant
basic forms on $NX$ which are vertically compactly supported.
Integration over the fibres of $NX$ gives a morphism of complexes
\[
\pi_*\colon\Omega_{\g,\cv}(E,\F_E)[r]\longto\Omega_\g(M,\F).
\]
An \emph{equivariant basic Thom form} of $NX$ is an $r$-form
$\tau_X\in\Omega_{\g,\cv}^r(NX,\F_{NX})$ which satisfies
$\pi_*\tau_X=1$ and $d_\g\tau_X=0$.  Such forms exists
by~\cite[Proposition~5.2.1]{lin-sjamaar;thom}.  The equivariant basic
Thom isomorphism theorem, \cite[Theorem~4.4.1]{lin-sjamaar;thom},
states that $\pi_*$ is a homotopy equivalence with homotopy inverse
equal to the map
\[
\zeta_*\colon\Omega_\g(M,\F)\longto\Omega_{\g,\cv}(E,\F_E)[r]
\]
defined by $\zeta_*(\alpha)=\tau_X\wedge\pi^*\alpha$.  Let
$f=f_X\colon NX\to M$ be an equivariant foliate tubular neighbourhood
embedding as in Proposition~\ref{proposition;tubular}.  The
\emph{pushforward homomorphism}
\begin{equation}\label{equation;pushforward}
i_*=i_{X,*}\colon\Omega(X,\F_X)\longto\Omega(M,\F)[r]
\end{equation}
is defined by $i_*=f_*\circ\zeta_*$,
i.e.\ $i_*\beta=f_*(\tau_X\wedge\pi_X^*\beta)$.  Here for each
$\gamma\in\Omega_\cv(NX,\F_{NX})$ the form $f_*(\gamma)\in\Omega(M)$
is defined by extending the form
$(f^{-1})^*(\gamma)\in\Omega(f(NX),\F)$, which is supported near $X$,
by zero.  The properties of the pushforward map $i_*$ include an
equivariant basic long exact Thom-Gysin sequence
(see~\cite[Proposition~5.3.3]{lin-sjamaar;thom}), as well as the next
lemma, which says that the map in cohomology induced by $i_*$ is
$H_\g(M,\F)$-linear and is independent of the choice of the Thom form
and the tubular neighbourhood embedding.

\begin{lemma}\label{lemma;pushforward}
Let $(M,\F,g)$ be a complete Riemannian foliated manifold equipped
with an isometric transverse Lie algebra action $\g\to\X(M,\F,g)$.
Let $X$ be a co-oriented $\g\ltimes\F$-invariant closed submanifold of
$M$ of codimension~$r$.
\begin{enumerate}
\item\label{item;pushforward-gdgm}
The pushforward homomorphism $i_*$ is a degree $r$ morphism of
$\g$-differential graded modules.  Reversing the co-orientation of $X$
changes the sign of $i_*$.
\item\label{item;pushforward-choice}
A different choice of equivariant basic Thom form $\tau_X$ or of
tubular neighbourhood embedding $f$ affects $i_*$ by a homotopy of
$\g$-differential graded modules.  It follows that the maps in
cohomology $i_*\colon H(X,\F_X)\to H(M,\F)[r]$ and $i_*\colon
H_\g(X,\F_X)\to H_\g(M,\F)[r]$ induced by $i_*$ are independent of the
choice of $\tau_X$ and $f$.
\item\label{item;linear}
The degree $r$ morphism of $\g$-differential graded modules
$\Omega(X,\F_X)\otimes\Omega(M,\F)\to\Omega(M,\F)[r]$ defined by
\[
\beta\otimes\alpha\longmapsto i_*\beta\wedge\alpha-i_*(\beta\wedge
i^*\alpha)
\]
is null-homotopic.  It follows that $i_*\colon H(X,\F_X)\to
H(M,\F)[r]$ is $H(M,\F)$-linear and that $i_*\colon H_\g(X,\F_X)\to
H_\g(M,\F)[r]$ is $H_\g(M,\F)$-linear.
\end{enumerate}
\end{lemma}

\begin{proof}
\eqref{item;pushforward-gdgm}~That $i_*=f_*\circ\zeta_*$ is morphism
of $\g$-dgm follows from the fact that $\tau_X$ is $\g$-equivariant.
If we equip $NX$ with the opposite orientation, the Thom form changes
to $-\tau_X$, so the pushforward map changes to $-i_*$.

\eqref{item;pushforward-choice}~Let $\tau_X'$ be another equivariant
basic Thom form for $NX$ and let $f'\colon NX\to M$ be another
equivariant foliate tubular neighbourhood embedding.  Then we have a
second Thom map $\zeta_*'\beta=\tau_X'\wedge\pi_X^*\beta $ and a
second pushforward morphism $i_*'=f_*'\circ\zeta_*'$, and we must show
that $i_*'$ is homotopic to $i_*$.  It suffices to show that
$\zeta_*'$ is homotopic to $\zeta_*$ and that $f_*'$ is homotopic to
$f_*$.  It follows from the equivariant basic Thom isomorphism
theorem, \cite[Theorem~4.4.1]{lin-sjamaar;thom}, that
$\tau_X'-\tau_X=d\upsilon$ for some form
$\upsilon\in\Omega_{\cv,\g}^{r+1}(NX,\F_{NX})$, so
\[l(\tau_X')-l(\tau_X)=[d,l(\upsilon)],\]
where ``$l$'' means left multiplication.  This implies
\[
\zeta_*'-\zeta_*=[d,l(\upsilon)\circ\pi^*],
\]
which shows that $\zeta_*'$ is homotopic to $\zeta_*$.  It follows
from Proposition~\ref{proposition;tubular} that there is an
equivariant foliate isotopy $F\colon[0,1]\times NX\to M$ from $f$ to
$f'$.  The track of this isotopy is an embedding
$\hat{F}\colon[0,1]\times NX\to[0,1]\times M$, so we have an extension
by zero map
\[
\hat{F}_*\colon\Omega([0,1]\times
NX,*\times\F_{NX})\longto\Omega([0,1]\times M,*\times\F).
\]
Let $\pr_{NX}\colon[0,1]\times NX\to NX$ and $\pr_M\colon[0,1]\times
M\to M$ be the cartesian projections and define $\kappa$ to be the
composition of the maps
\[
\begin{tikzcd}[row sep=tiny]
\Omega(NX,\F_{NX})\ar[r,"\pr_{NX}^*"]&
\Omega([0,1]\times NX,*\times\F_{NX})
\\
\qquad\qquad\ar[r,"\hat{F}_*"]&
\Omega([0,1]\times M,*\times\F)\ar[r,"\pr_{M,*}"]&\Omega(M,\F)[-1].
\end{tikzcd}
\]
Define $j_t\colon M\to[0,1]\times M$ by $j_t(x)=(t,x)$.  Then
$[d,\pr_{M,*}]=j_1^*-j_0^*$ by Stokes' formula; see
e.g.\ \cite[Corollary~B.4]{lin-sjamaar;thom}.  It follows that
\[
[d,\kappa]=[d,\pr_{M,*}]\circ\hat{F}_*\circ\pr_{NX}^*=
(j_1^*-j_0^*)\circ\hat{F}_*\circ\pr_{NX}^*=f'_*-f_*,
\]
which shows that $f_*'$ is homotopic to $f_*$.

\eqref{item;linear}~For $\alpha\in\Omega(M,\F)$ and
$\beta\in\Omega(X,\F_X)$ we have
\[
i_*\beta\wedge\alpha-i_*(\beta\wedge i^*\alpha)=
f_*\bigl(\zeta_*\beta\wedge(\id_{NX}^*-\pi^*\zeta^*)f^*\alpha\bigr).
\]
The map $\zeta\circ\pi\colon NX\to NX$ is homotopic to the identity
via a $\g$-equivariant foliate homotopy.  Let
$\kappa\colon\Omega(NX,\F_{NX})\to\Omega(NX,\F_{NX})[-1]$ be a
homotopy from $\id_{NX}^*$ to $\pi^*\zeta^*$ as in
Lemma~\ref{lemma;cochain-homotopy}.  Define
\[\bar\kappa\colon\Omega(NX,\F_{NX})\to\Omega(NX,\F_{NX})[r-1]\]
by $\bar\kappa(\beta\otimes\alpha)=
(-1)^{r+\abs{\beta}}f_*(\zeta_*\beta\wedge\kappa f^*\alpha)$.  A
calculation shows that $[\iota(\xi),\bar\kappa]=[L(\xi),\bar\kappa]=0$
for all $\xi\in\g$ and that
\[
\begin{aligned}
i_*\beta\wedge\alpha-i_*(\beta\wedge i^*\alpha)&=
f_*(\zeta_*\beta\wedge(d\kappa+\kappa d)f^*\alpha)\\
&=d\bar\kappa\bigl(\beta\otimes\alpha)+
(-1)^r\bar\kappa(d\beta\otimes\alpha+(-1)^{\abs{\beta}}\beta\otimes
d\alpha\bigr)\\
&=[d,\bar\kappa](\beta\otimes\alpha).
\end{aligned}
\]
Therefore $\bar\kappa$ is a homotopy of $\g$-dgm of degree $r-1$.
\end{proof}

Let us call a subset of $M$ \emph{transversely compact} if it is
closed and its image in the leaf closure space $M/\barF$ is compact.
A subset $C$ of $M$ is transversely compact if and only if the
corresponding subset $\varrho(\pi^{-1}(C))$ of the Molino manifold $W$
is compact.  We denote the family of all transversely compact subsets
by ``$\ct$''.  The basic differential forms with transversely compact
supports constitute a subcomplex $\Omega_\ct(M,\F)$ of the basic de
Rham complex $\Omega(M,\F)$.
\glossary{ct@$\ct$, transversely compact supports}

For simplicity we will from now on assume that the foliation $\F$ is
transversely oriented and that the top exterior power $\det(\liec)$ of
the Molino centralizer bundle is a trivial flat line bundle.  This
assumption implies that the Molino manifold $W$ is orientable.
Without this assumption the following definition is not quite correct,
but needs to be amended by using forms with coefficients in the
homological orientation sheaf of~\cite{sergiescu;cohomologie-basique}.

\begin{definition}\label{definition;integral}
A \emph{transverse integral} $\fint$ is a collection of $\R$-linear
functionals
\[\fint_X\colon\Omega_\ct(X,\F_X)\to\R,\]
one for each co-oriented $\F$-invariant closed submanifold $X$ of $M$,
which satisfy the following conditions:
\begin{enumerate}
\renewcommand\theenumi{\alph{enumi}}
\renewcommand\labelenumi{(\theenumi)}
\item\label{item;component}
$\fint_M\alpha=\fint_M\alpha_q$ for all $\alpha\in\Omega_\ct(M,\F)$,
  where $\alpha_q$ is the component of $\alpha$ of degree equal to
  $q=\dim(M)-\dim(\F)$;
\item\label{item;test}
for $\alpha\in\Omega(M,\F)$ we have $\alpha=0$ if and only if
$\fint_M\alpha\wedge\beta=0$ for all $\beta\in\Omega_\ct(M,\F)$;
\item\label{item;transverse-stokes}
Stokes' formula holds: $\fint_Md\alpha=0$ for all
$\alpha\in\Omega_\ct(M,\F)$;
\item\label{item;poincare}
Poincar\'e duality holds: the pairing
\[H(M,\F)\times H_\ct(M,\F)\longto\R\]
defined by $([\alpha],[\beta])\mapsto\fint_M\alpha\wedge\beta$ induces
an isomorphism $H(M,\F)[q]\to H_\ct(M,\F)^*$; and
\item\label{item;fundamental}
the form $i_{X,*}1=f_{X,*}\tau_X\in\Omega(M,\F)$ represents the basic
fundamental class of $X$ in the sense that
$\fint_Xi_X^*\alpha=\fint_Mi_{X,*}1\wedge\alpha$ for all
$\alpha\in\Omega_\ct(M,\F)$, where $i_{X,*}$ denotes the pushforward
map~\eqref{equation;pushforward}.
\end{enumerate}
\end{definition}

\begin{remark}\label{remark;integral}
It follows from
Lemma~\ref{lemma;pushforward}\eqref{item;pushforward-choice} and
Stokes' formula~\eqref{item;transverse-stokes} that the integral
$\fint_Mi_{X,*}1\wedge\alpha$ is independent of the choice of the Thom
form $\tau_X$ and the embedding $f_X$.  However, reversing the
co-orientation of $X$ changes the sign of $i_*$
(Lemma~\ref{lemma;pushforward}\eqref{item;pushforward-gdgm}) and
therefore, by condition~\eqref{item;fundamental}, changes the sign of
the transverse integral $\fint_X$.
\end{remark}

Transverse integrals have the following elementary properties.

\begin{lemma}\label{lemma;transverse}
Let $(M,\F,g)$ be a complete Riemannian foliated manifold equipped
with an isometric transverse Lie algebra action $\g\to\X(M,\F,g)$.
Let $\fint$ be a transverse integral and let $i=i_X\colon X\to M$ be a
co-oriented $\F$-invariant closed submanifold of $M$.
\begin{enumerate}
\item\label{item;adjunction}
The homomorphisms $i^*$ and $i_*$ are adjoint in the sense that
\[\fint_X\beta\wedge i^*\alpha=\fint_M i_*\beta\wedge\alpha\]
for all $\alpha\in\Omega_\ct(M,\F)$ and $\beta\in\Omega(X,\F_X)$.
\item\label{item;integrals}
Let $\fint'$ be another transverse integral.  There exists a nonzero
constant $k$, independent of $X$, such that
$\fint_X'\beta=k\fint_X\beta$ for all closed
$\beta\in\Omega_\ct(X,\F_X)$.
\item\label{item;transverse-invariant}
Let $\beta\in\Omega_\ct(X,\F_X)$.  Then
$\fint_XL(v)\beta=\fint_X\iota(v)\beta=0$ for all transverse Killing
vector fields $v\in\X(X,\F_X,g)$.
\item\label{item;equivariant-stokes}
Suppose $X$ is $\g\ltimes\F$-invariant.  The $\W\g$-linear functional
\[
\begin{tikzcd}
\fint_X=\id_{\W\g}\otimes\fint_X\colon\W\g\otimes\Omega_\ct(X,\F_X)
\ar[r]&\W\g
\end{tikzcd}
\]
is a morphism of $\g$-differential graded modules of degree
$-\codim(\F_X)$.  In particular $\fint_X\beta\in(S\g^*)^\g$ and
$\fint_Xd_\g\beta=0$ for all $\beta$ in the Weil complex
$\Omega_{\g,\ct}(X,\F_X)$.
\end{enumerate}
\end{lemma}

\begin{proof}
\eqref{item;adjunction}~This follows from
Definition~\ref{definition;integral}\eqref{item;fundamental} and the
fact that $i_*\beta=i_*1\wedge f_*\pi_X^*\beta$.

\eqref{item;integrals}~By Poincar\'e duality
(Definition~\ref{definition;integral}\eqref{item;poincare}), each of
the transverse integrals $\fint_M$ and $\fint_M'$ gives an isomorphism
$H_\ct^q(M,\F)^*\cong H^0(M,\F)\cong\R$.  Hence there exists a nonzero
constant $k$ such that $\fint_M'\alpha=k\fint_M\alpha$ for all
$\alpha\in\Omega_\ct^q(M,\F)$.  It follows
from~\eqref{item;adjunction} that $\int_X\beta=\int_Mi_*\beta$ and
$\int_X'\beta=\int_M'i_*\beta$ for all $\beta\in\Omega_\ct(X)$.
Therefore $\fint_X'\beta=c\fint_X'\beta$ for all
$\beta\in\Omega_\ct(X)$.

\eqref{item;transverse-invariant}~Using $\int_X\beta=\int_Mi_*\beta$
and Definition~\ref{definition;integral}\eqref{item;component} we get
$\fint_X\beta=\fint_X\beta_{q_X}$ for all
$\beta\in\Omega_\ct(X,\F_X)$, where $\beta_{q_X}$ is the component of
$\beta$ of degree equal to $q_X=\dim(X)-\dim(\F)$.  This implies the
identity $\fint_X\iota(v)\beta=0$, because a basic form $\beta$ on $X$
of degree $q_X+1$ is equal to~$0$.  We also have
$\fint_Xd\beta=\fint_Mi_*d\beta=\fint_Mdi_*\beta=0$ because of Stokes'
formula
(Definition~\ref{definition;integral}\eqref{item;transverse-stokes}).
Hence $\fint_XL(v)\beta=\fint_X(d\iota(v)\beta+\iota(v)d\beta)=0$.

\eqref{item;equivariant-stokes}~That $\fint_M$ is a morphism of
$\g$-differential modules follows
from~\eqref{item;transverse-invariant}.  By definition
$\alpha\in\Omega_{\g,\ct}(M,\F)$ is a $\g$-basic element of
$\W\g\otimes\Omega_{\g,\ct}(M,\F)$.  Hence
$\fint_M\alpha\in(\W\g)_\bbas{\g}=(S\g^*)^\g$.  The differential of
the Weil algebra $d_{\W\g}$ vanishes on basic elements (see
e.g.\ \cite[Proposition~A.4.2(iii)]{lin-sjamaar;thom}), so
$\fint_Md_\g\alpha=d_{\W\g}\fint_M\alpha=0$.
\end{proof}

We are aware of two approaches to transverse integration.  The first,
which works if $M$ is compact, is that of Kamber and Tondeur
(see~\cite[Ch.~7]{tondeur;geometry-foliations}
or~\cite[\S\,3]{goertsches-nozawa-toeben;chern-simons-foliations}) and
the second is that of Sergiescu~\cite{sergiescu;cohomologie-basique}.
We review Sergiescu's approach and show that it defines a transverse
integral in our sense.

Let $\nu_{\liek}$ be the volume element on the Lie algebra
$\liek=\lie{o}(q)$ defined by the normalized invariant inner product
$g_{\liek}$.  Let $\theta_\LC$ be the transverse Levi-Civita
connection of the Molino bibundle $P$
(Section~\ref{section;molino-diagram}).  Then
$\nu=\theta_\LC^*(\nu_{\liek})$ is an $\F_P$-basic form on $P$ that
restricts to a normalized invariant volume form on each fibre of the
bundle $\pi\colon P\to M$.  The projection formula gives
$\pi_*(\nu\wedge\pi^*\alpha)=\pi_*\nu\wedge\alpha=\alpha$ for all
forms $\alpha$ on $M$, which means that the map
\[
\Omega(M,\F)\longto\Omega(P,\F_P)\colon\quad\alpha\longmapsto
\nu\wedge\pi^*\alpha
\]
is a right inverse of the fibre integration map $\pi_*$.  Choose a
nonzero global flat section $s$ of the real line bundle $\det(\liec)$,
where $\liec$ is the centralizer bundle of $(M,\F)$ defined in
Theorem~\ref{theorem;molino-2}.  For every $\gamma\in\Omega(P,\F_P)$
the contraction $\iota_{\pi^\dag(s)}\gamma$ with the multivector field
$\pi^\dag(s)$ is well-defined.  Every $\F_P$-basic form is
$\barF_P$-invariant, so $\iota_{\pi^\dag(s)}\gamma$ is $\barF_P$-basic
and therefore equal to $\varrho^*\mu$ for a unique form $\mu$ on $W$,
which we will denote by $\mu=\varrho_\dag\iota_{\pi^\dag(s)}\gamma$.
Let $X$ be a co-oriented $\g\ltimes\F$-invariant closed submanifold of
$M$ of codimension~$r$.  Then the corresponding $G\times K$-invariant
submanifold $X_W=\varrho(\pi(X))$ of the Molino manifold $W$ is
oriented.  For simplicity let us write $\pi$ for the restriction of
$\pi$ to $X_P=\pi^{-1}(X)$, and similarly for $\varrho$, $\nu$, and
$s$.  Then for $\beta\in\Omega_\ct(X,\F)$ we define
\begin{equation}\label{equation;integral}
\fint_X\beta=\int_{X_W}\varrho_\dag\iota_{\pi^\dag(s)}
(\pi^*\beta\wedge\nu).
\end{equation}
Here we adopt the usual convention that $\int_{X_W}\mu$ is the
integral of the component of $\mu$ of degree $\dim(X_W)$ If $\beta$ is
homogeneous of degree $j$, then the degree of
$\varrho_\dag\iota_{\pi^\dag(s)}(\pi^*\beta\wedge\nu)$ is
$j+\frac12q(q-1)-l$.  Using $\dim(X_W)=q_X+\frac12q(q-1)-l$, where
$q_X=\dim(X)-\dim(\F)$ and $l$ is the rank of the centralizer bundle
$\liec$ (see table of dimensions at the end of the notation index), we
see that $\fint_X\beta$ is the transverse integral of the component of
$\beta$ of degree $q_X=\dim(X)-\dim(\F)$.  In particular the
transverse integral $\fint_X1$ of $1\in\Omega^0(M,\F)$ is defined if
$X$ is a closed leaf of $\F$.

\begin{proposition}\label{proposition;integral}
Let $(M,\F,g)$ be a complete Riemannian foliated manifold equipped
with an isometric transverse Lie algebra action $\g\to\X(M,\F,g)$.
Suppose that the foliation $\F$ is transversely oriented and that the
line bundle $\det(\liec)$ is trivial.  Then the
formula~\eqref{equation;integral} defines a transverse integral in the
sense of Definition~\ref{definition;integral}.
\end{proposition}
 
\begin{proof}
Conditions~\eqref{item;component}--\eqref{item;poincare} of
Definition~\ref{definition;integral} are verified
in~\cite[\S\,2]{sergiescu;cohomologie-basique}.
Condition~\eqref{item;fundamental} is checked as follows.  By
definition
\[
\fint_Mi_*1\wedge\alpha=\fint_Mf_*\tau_X\wedge\alpha=
\int_W\varrho_\dag\iota_{\pi^\dag(s)}
\bigl(\pi^*(f_*\tau_X\wedge\alpha)\wedge\nu\bigr).
\]
In this formula we may take $\tau_X$ to be any equivariant basic Thom
form for the normal bundle $NX$.  Let us take $\tau_X$ to be the
universal equivariant basic Thom form defined
in~\cite[\S\,4.5]{lin-sjamaar;thom}.  Then $\tau_X$ is
$\barF_{NX}$-basic by~\cite[Lemma~4.5.2(iv)]{lin-sjamaar;thom}, so
$\iota_{\pi^\dag(s)}\pi^*f_*(\tau_X)=\pi^*f_*\iota_s(\tau_X)=0$.
Under the Molino correspondence $M\leftrightarrow W$ the submanifold
$X$ corresponds to the submanifold $X_W=\varrho(\pi^{-1}(X))$ and we
have an isomorphism of metric vector bundles with connection
$\pi^*NX\cong\varrho^*NX_W$.  (See Remark~\ref{remark;tubular}.)  The
universality property of the Thom form
(\cite[Lemma~4.5.2(ii)]{lin-sjamaar;thom}) gives
$\pi^*(\tau_X)=\varrho^*(\tau_{X_W})$, where $\tau_{X_W}$ is the
universal Thom form of the bundle $NX_W$.  Hence
\begin{align*}
\fint_Mi_*1\wedge\alpha&=\int_W\varrho_\dag\iota_{\pi^\dag(s)}
\bigl(\pi^*(f_*\tau_X\wedge\alpha)\wedge\nu\bigr)\\
&=\int_Wf_{W,*}\tau_{X_W}\wedge
\varrho_\dag\iota_{\pi^\dag(s)}\bigl(\pi^*\alpha\wedge\nu\bigr)\\
&=\int_{X_W}i_{X_W}^*\varrho_\dag\iota_{\pi^\dag(s)}
\bigl(\pi^*\alpha\wedge\nu\bigr)\\
&=\fint_Xi^*\alpha,
\end{align*}
where we used the fact that the Thom form $\tau_{X_W}$ represents the
Poincar\'e dual of the submanifold $X_W$; see e.g.~\cite[Proposition
  6.24]{bott-tu;differential-forms}.
\end{proof}

We close this section by reiterating a point made by T\"oben.  Suppose
the submanifold $X$ has the property that all leaves of $\F_X$ are
closed.  Then the leaf space $X/\F_X=X/\barF_X=X_W/K$ is an orbifold
(Proposition~\ref{proposition;leaf-orbifold}), and a basic
differential form $\beta$ on $X$ represents a differential form on the
orbifold.  The transverse integral $\fint_X\beta$ relates to the
orbifold integral $\int_{X/\F_X}\beta$ as follows.  The
\emph{principal stratum} of $X$ is the open dense subset $X^\prin$
consisting of all $x\in X$ for which the holonomy of $\F_X$ vanishes
at $x$.  The corresponding submanifold $\varrho(\pi(X^\prin))$ of the
Molino manifold $W$ is equal to the principal orbit type stratum of
$X_W$ considered as a $K$-manifold.

\begin{proposition}[{\cite[Lemma
    5.4]{toeben;localization-basic-characteristic}}]
\label{proposition;orbifold-integration}
Under the hypotheses of Proposition~\ref{proposition;integral}, let
$X$ be a co-oriented $\F$-invariant closed submanifold of $M$.
Suppose that the leaves of $\F_X$ are closed.  Then the transverse
integral $c=\fint_{\F(x)}1$ is independent of $x\in X^\prin$ and for
all $\beta\in\Omega_\ct(X,\F_X)$ we have
\[
\fint_X\beta=(-1)^{\frac12q(q+1)}c\int_{X/\F_X}\beta.
\]
\end{proposition}

\subsection{Euler forms}\label{section;euler}

Let $(M,\F,g)$ be a complete Riemannian foliated manifold equipped
with an isometric transverse action of a Lie algebra $\g$.  By
Theorem~\ref{theorem;symmetry}\eqref{item;fixed} the fixed-leaf
manifold $M^\g$ is a $\g\ltimes\F$-invariant closed submanifold of
$M$.  We now show that if $\g$ is abelian every connected component
$X$ of $M^\g$ is co-oriented and of even codimension $2l$.  Thus each
$X$ has an equivariant basic Thom form
$\tau_X\in\Omega_{\cv,\g}^{2l}(NX,\F_{NX})$.  Its restriction to the
zero section $\eta_X=\zeta_X^*\tau_X\in\Omega^{2l}(X,\F_X)$ is known
as the associated \emph{equivariant basic Euler form}.  The next
result, which is well known for torus actions on manifolds (see
e.g.~\cite[\S\,3]{atiyah-bott;moment-map-equivariant-cohomology}
or~\cite[\S\,1]{brion;equivariant-cohomology-intersection-theory}),
gives a simple formula for the restriction of the Euler form to a
point.

\begin{proposition}\label{proposition;euler-fixed}
Let $(M,\F,g)$ be a complete Riemannian foliated manifold equipped
with an isometric transverse action of an abelian Lie algebra $\g$ and
let $X$ be a connected component of the fixed-leaf manifold $M^\g$.
Let $r$ be the codimension of $X$ in $M$ and let $E=NX$ be the normal
bundle of $X$.  Let $j_x\colon\{x\}\to X$ be the inclusion of $x\in
X$, and let $a_x\colon\g\to\lie{o}(E_x,g_{E,x})$ be the $\g$-action on
the fibre $E_x$.
\begin{enumerate}
\item\label{item;complex}
The bundle $E$ has a $\g$-invariant almost complex structure $J$.  The
weights $\lambda_1$, $\lambda_2$,~\dots, $\lambda_l\in\g^*$ of the
action $a_x$ with respect to $J_x$ are nonzero and are independent of
$x\in X$.  We have a weight space decomposition $E=E_{\lambda_1}\oplus
E_{\lambda_2}\oplus\cdots\oplus E_{\lambda_l}$ into $\g$-equivariant
foliated subbundles.  In particular $E$ is orientable and the rank
$r=2l$ of $E$ is even.
\item\label{item;point}
Let $\eta_X=\zeta_X^*\tau_X\in\Omega_\g^r(X,\F_X)$ be an equivariant
basic Euler form with respect to the orientation given by $J$ and let
$\eta_0$ be the component of $\eta_X$ in
$S\g^*\otimes\Omega^0(X,\F_X)$.  Then
$\eta_0=\lambda_1\lambda_2\cdots\lambda_l\in S^l\g^*$.  Reversing the
co-orientation of $X$ has the effect of changing the sign of $\eta_X$
and $\eta_0$.
\item\label{item;nilpotent}
The Euler form $\eta_X$ becomes invertible in the algebra
$\Omega_\g(X,\F_X)$ after inverting the weights $\lambda_1$,
$\lambda_2$,~\dots, $\lambda_l$.
\end{enumerate}
\end{proposition}

\begin{proof}
\eqref{item;complex}~The transverse $\g$-action on $X$ being trivial,
the transverse $\g$-action on $E$ is tangent to the fibres of $E$.
The foliation induced by $\F_E$ on each fibre $E_x$ is the trivial
$0$-dimensional foliation, so on $E_x$ we have a linear $\g$-action in
the usual sense.  The fibre metric of $E$ is by assumption
$\g$-invariant, so for each $x\in X$ the transverse $\g$-action on
$E_x$ is given by an infinitesimal orthogonal representation
$a_x\colon\g\to\lie{o}(E_x,g_{E,x})$.  Our vector bundle comes
equipped with an invariant metric connection, so a parallel transport
argument shows that the representation
$a_y\colon\g\to\lie{o}(E_y,g_{E,y})$ is equivalent to $a_x$ for all
$y\in X$.  Since $a_x$ leaves only the origin of $E_x$ fixed, we have
a real isotypical decomposition
\begin{equation}\label{equation;isotype}
E_x\cong\bigoplus_{\lambda\in P}\R^2(\lambda)^{\oplus m(\lambda)}.
\end{equation}
Here $\R^2(\lambda)$ is a copy of $\R^2$ on which $\g$ acts through
the infinitesimal rotations
\begin{equation}\label{equation;plane-rotation}
\xi\longmapsto2\pi\lambda(\xi)\begin{pmatrix}0&-1\\1&0\end{pmatrix},
\end{equation}
$m(\lambda)$ is the multiplicity of $\R^2(\lambda)$, and $P$ is a
finite subset of $\g^*\backslash\{0\}$.  Since
$\R^2(\lambda)\cong\R^2(-\lambda)$, the weights $\lambda$ are defined
only up to sign, so $P$ is not unique.  To remove the ambiguity we
choose an element $\xi_0\in\g$ which fixes only the origin of $E_x$
and require all $\lambda\in P$ to satisfy $\lambda(\xi_0)>0$.  We give
$E_x$ the almost complex structure $J_x$ which is equal to
$\bigl(\begin{smallmatrix}0&-1\\1&0\end{smallmatrix}\bigr)$ on
  $\R^2(\lambda)$.  This almost complex structure depends on the
  choice of $\xi_0$, but not on the
  isomorphism~\eqref{equation;isotype}
  (see~\cite[\S\,8.5.3]{guillemin-sternberg;supersymmetry-equivariant}),
  and gives us the desired almost complex structure $J$ and the
  corresponding isotypical decomposition of $E$.

\eqref{item;point}~Since $\g$ acts trivially on $X$, the Weil complex
of $X$ is $\Omega_\g(X,\F_X)=S\g^*\otimes\Omega(X,\F_X)$ with
differential $d_\g={\id}\otimes d$.  Because the Euler form $\eta_X$
is closed, this shows that its component $\eta_0\in
S^l\g^*\otimes\Omega^0(X,\F_X)$ is closed.  In other words, for any
$x\in X$, $\eta_0\in S^l\g^*$ is equal to $j_x^*(\eta_X)$, which is a
$\g$-equivariant Euler form for the bundle $E_x$ over the one-point
space $\{x\}$.  To calculate this, we may take $\tau_X$ to be the
universal equivariant basic Thom form
of~\cite[\S\,4.5]{lin-sjamaar;thom} with respect to a $\g$-invariant
basic metric connection $\theta$ on $E$.  Then the Euler form is
\[
\eta_X=(-2\pi)^{-l}c_{\g,\theta}(\Pf),
\]
where $c_{\g,\theta}\colon(S\liek^*)^{\liek}\longto\Omega_\g(X,\F_X)$
is the $\g$-equivariant characteristic, or Cartan-Chern-Weil,
homomorphism~\eqref{equation;equivariant-characteristic}, with
$\liek=\lie{o}(2l)$, and $\Pf\in S^l(\liek^*)^\liek$ is the Pfaffian.
The $\g$-equivariant characteristic homomorphism
$(S\liek^*)^{\liek}\to\Omega_\g(\{x\})\cong S\g^*$ for the bundle
$E_x\to\{x\}$ is just the restriction map
$a_x^*\colon(S\liek^*)^{\liek}\to S\g^*$ induced by the action
$a_x\colon\g\to\liek$, so the naturality property of the universal
Euler form (see~\cite[Proposition~4.6.1]{lin-sjamaar;thom}) with
respect to the map $j_x$ yields
$j_x^*(\eta_X)=(-2\pi)^{-l}a_x^*(\Pf)$.  The polynomial $a_x^*(\Pf)$
is the product of the Pfaffians of the
matrices~\eqref{equation;plane-rotation},
i.e.\ $a_x^*(\Pf)=(-2\pi)^l\lambda_1\lambda_2\cdots\lambda_l$, and
hence $j_x^*(\eta_X)=\lambda_1\lambda_2\cdots\lambda_l$.  The last
assertion follows from the fact that if we reverse the orientation of
$E$ the Thom form changes sign.

\eqref{item;nilpotent}~The form $\eta_X-\eta_0\in
S\g^*\otimes\Omega^+(X,\F_X)$ is nilpotent.  It follows from
\eqref{item;point} that $\eta_0=\lambda_1\lambda_2\cdots\lambda_l$ is
a nonzero element of $S^l\g^*$.  Therefore $\eta_X$ is invertible in the
localization of the $S\g^*$-algebra $\Omega_\g(X,\F_X)$ at the weights
$\lambda_i$.
\end{proof}

The utility of~\eqref{item;nilpotent} lies in the fact that the
localization map
\[
\Omega_\g(X,\F_X)\to
\Omega_\g(X,\F_X)\bigl[(\lambda_1\cdots\lambda_l)^{-1}\bigr]
\]
is injective, because the weights in $\g^*\backslash\{0\}$ are not
zero divisors in $S\g^*$.  Identities involving equivariant basic
forms can therefore be verified in the localized algebra, which is
easier to handle.

\subsection{Covariant localization}\label{section;covariant}

Let $(M,\F,g)$ be a complete Riemannian foliated manifold equipped
with an isometric transverse action of an abelian Lie algebra $\g$.
By Proposition~\ref{proposition;euler-fixed} connected components $X$
of the fixed-leaf manifold $M^\g$ are co-orientable and of even
codimension $2l_X$.  Fixing a co-orientation for each $X$ we have
pushforward homomorphisms $i_{X,*}$, which we can assemble to a single
map
\[
i_{\g,*}=\bigoplus_Xi_{X,*}\colon\bigoplus_X\Omega_\g(X,\F_X)[-2l_X]
\longto\Omega_\g(M,\F),
\]
where $X$ ranges over the connected components of $M^\g$.  For each
$X$ the normal bundle $NX$ has an equivariant basic Euler form
$\eta_X\in\Omega_\g^{2l_X}(X,\F)$ and a well-defined weight product
$\lambda_X\in S^{l_X}\g^*$.  By
Proposition~\ref{proposition;euler-fixed}\eqref{item;nilpotent} the
forms $\eta_X$ become invertible after extending the ring of scalars
$S\g^*$ by inverting the $\lambda_X$, so the following statement makes
sense.

\begin{theorem}[covariant localization]\label{theorem;covariant}
Let $(M,\F,g)$ be a transversely compact complete Riemannian foliated
manifold equipped with an isometric transverse action of an abelian
Lie algebra $\g$.  The $S\g^*$-linear homomorphism
$i_{\g,*}=\bigoplus_Xi_{X,*}$ induces an isomorphism on cohomology
after inverting the weight products $\lambda_X\in S\g^*$ and any $f\in
S\g^*$ that vanishes on the cone $\Gamma_M=\bigcup_{x\in M\backslash
  M^\g}\stab(\g\ltimes\F,x)$.  The inverse of this isomorphism is
induced by the map $i_\g^!$ defined by
\[
i_\g^!(\alpha)=\sum_X\eta_X^{-1}\wedge i_X^*\alpha,
\]
where the sum is over the connected components $X$ of the fixed-leaf
manifold $M^\g$.
\end{theorem}

\begin{proof}
For all $\beta\in\Omega_\g(M^\g,\F_{M^\g})$ and for all components $X$
of $M^\g$ we have
\[
i_X^*i_{X,*}\beta=\zeta_X^*\zeta_{X,*}\beta=
\zeta_X^*\bigl(\tau_\g(NX,\F_{NX})\wedge\pi_X^*\beta\bigr)=\eta_X\wedge
i_X^*\beta,
\]
and therefore
\[
i_\g^!(i_{\g,*}\beta)=\sum_X\eta_X^{-1}\wedge
i_X^*i_{X,*}\beta=\sum_X\eta_X^{-1}\wedge\eta_X\wedge
i_X^*\beta=\sum_Xi_X^*\beta=\beta,
\]
which shows that $i_\g^!\circ i_{\g,*}=\id$.  Let
$i_\g^*=\sum_Xi_X^*\colon\Omega_\g(M,\F)\to\Omega_\g(M^\g,\F_{M^\g})$
be the restriction map.  For all $\alpha\in\Omega_\g(M,\F)$ we have
\[
i_\g^*i_{\g,*}i_\g^!(\alpha)= \sum_Xi_X^*i_{X,*}(\eta_X^{-1}\wedge
i_X^*\alpha)=\sum_X\eta_X\wedge\eta_X^{-1}\wedge
i_X^*\alpha=\sum_Xi_X^*\alpha=i_\g^*\alpha,
\]
which shows that $i_\g^*\circ i_{\g,*}\circ i_\g^!=i_\g^*$.  By the
contravariant localization theorem, Theorem~\ref{theorem;borel},
$i_\g^*$ induces an isomorphism on cohomology, so we conclude that
$i_{\g,*}\circ i_\g^!$ induces the identity on cohomology.
\end{proof}

Pushing forward to a point now gives an integration formula, which is
a foliated version of the formulas of Atiyah and
Bott~\cite{atiyah-bott;moment-map-equivariant-cohomology} and Berline
and Vergne~\cite{berline-vergne;zeros}.

\begin{theorem}\label{theorem;abbv} 
Let $(M,\F,g)$ be a transversely compact complete Riemannian foliated
manifold equipped with an isometric transverse action of an abelian
Lie algebra $\g$.  Suppose that the foliation $\F$ is transversely
oriented and that the line bundle $\det(\liec)$ is trivial.  Then
\[
\fint_M\alpha=\sum_X\fint_X\eta_X^{-1}\wedge i_X^*\alpha
\]
for all $d_\g$-closed $\alpha\in\Omega_\g(M,\F)$, where the sum is
over the connected components $X$ of the fixed-leaf manifold $M^\g$.
\end{theorem}

\begin{proof}
It follows from Theorem~\ref{theorem;covariant} and from Stokes'
formula that
\[
\fint_M\alpha=\fint_Mi_{\g,*}i_\g^!(\alpha)=
\sum_X\fint_Mi_{X,*}(\eta_X^{-1}\wedge i_X^*\alpha).
\]
Lemma~\ref{lemma;transverse}\eqref{item;adjunction} now yields the
desired result.
\end{proof}

\begin{remark}
Reversing the co-orientation of $X$ changes the sign of the integral
$\fint_X$ (Remark~\ref{remark;integral}) and of the Euler form
$\eta_X$
(Proposition~\ref{proposition;euler-fixed}\eqref{item;point}), so the
integral $\fint_X\eta_X^{-1}\wedge i_X^*\alpha$ is independent of the
choice of co-orientation.
\end{remark}

This immediately gives a foliated version of the Duistermaat-Heckman
formula~\cite{duistermaat-heckman;variation}, where we consider the
case of a \emph{transverse symplectic form}, i.e.\ a closed basic
$2$-form $\omega\in\Omega^2(M,\F)$ which is nondegenerate on the
normal bundle $N\F$.  Then the codimension $q=2n$ of $\F$ is even.
Suppose that the transverse $\g$-action is Hamiltonian with moment map
$\Phi\colon M\to\g^*$.  Then the \emph{equivariant transverse
  symplectic form} $\omega_\g=\omega+\Phi$ is $\g$-equivariant and
$\F$-basic.  The fixed-leaf components $X$ are transversely symplectic
of dimension $2n_X$, and the moment map has a constant value $\Phi_X$
on $X$.  Applying the integration formula Theorem~\ref{theorem;abbv}
to the form $\alpha=\exp(\omega_\g)$ gives the following result.

\begin{theorem}\label{theorem;duistermaat-heckman}
Let $(M,\F,g)$ be a transversely compact complete Riemannian foliated
manifold equipped with an isometric transverse action of an abelian
Lie algebra $\g$.  Suppose that $(M,\F)$ is transversely symplectic,
that the line bundle $\det(\liec)$ is trivial, and that the transverse
$\g$-action is Hamiltonian with moment map $\Phi\colon M\to\g^*$.
Then
\[
\fint_Me^\Phi\frac{\omega^n}{n!}=
\sum_Xe^{\Phi_X}\fint_X\eta_X^{-1}\wedge\frac{\omega^{n_X}}{n_X!}.
\]
\end{theorem}

See \cite[\S\,10]{hoffman-sjamaar;hamiltonian-stack} for a different
take on the foliated Duistermaat-Heckman theorem.

\input{riemannian.gls}

%
\begin{center}
\begin{tabular}{@{}ccc@{}}
\toprule
\multicolumn{3}{c}{Dimensions and ranks}\\
\midrule
$\dim(M)$&$=$&$m$\\
$\codim_M(\F)$&$=$&$q$\\
$\rank(\liec)$&$=$&$l$\\
$\dim(K)$&$=$&$\frac12q(q-1)$\\
$\dim(P)$&$=$&$m+\frac12q(q-1)$\\
$\dim(W)$&$=$&$q+\frac12q(q+1)-l$\\
$\codim_X(\F_X)$&$=$&$q_X$\\
$\dim(X_W)$&$=$&$q_X+\frac12q(q+1)-l$\\
\bottomrule
\end{tabular}
\end{center}
%


\bibliographystyle{amsplain}

\bibliography{hamilton}


\end{document}